\documentclass{article}

\title{Proximal-Proximal-Gradient Method} 
\author{Ernest K. Ryu and Wotao Yin}

\usepackage{graphicx} 
\usepackage{subfigure}

\usepackage{tikz}
\usetikzlibrary{arrows.meta, shapes.misc}
\usepackage{pgf}
\usetikzlibrary{arrows, automata}
\usetikzlibrary{arrows,calc,positioning}
\usepackage{amsmath,amssymb, amsthm}

\DeclareMathOperator*{\argmin}{argmin}
\newcommand{\prox}{{\mathbf{prox}}}

\newcommand{\reals}{\mathbb{R}}

\newcommand{\tnabla}{\tilde{\nabla}}
\newcommand{\EE}{\mathbb{E}}

\newtheorem{theorem}{Theorem}
\newtheorem{lemma}{Lemma}
\newtheorem{corollary}{Corollary}

\newcommand{\brr}{\bar{r}}
\newcommand{\brf}{\bar{f}}
\newcommand{\brg}{\bar{g}}

\newcommand{\vd}{\mathbf{d}}
\newcommand{\vx}{\mathbf{x}}

\newcommand{\vz}{\mathbf{z}}

\begin{document}

\maketitle
\begin{abstract}
In this paper, we present the proximal-proximal-gradient method (PPG), a novel optimization method that is simple to implement and simple to parallelize. PPG generalizes the proximal-gradient method and ADMM and is applicable to minimization problems written as a sum of many differentiable and many non-differentiable convex functions. The non-differentiable functions can be coupled. We furthermore present a related stochastic variation, which we call stochastic PPG (S-PPG). S-PPG can be interpreted as a generalization of Finito and MISO over to the sum of many coupled non-differentiable convex functions.

We present many applications that can benefit from PPG and S-PPG and prove convergence for both methods. We demonstrate the empirical effectiveness of both methods through experiments on a CUDA GPU. A key strength of PPG and S-PPG is, compared to existing methods, their ability to directly handle a large sum of non-differentiable non-separable functions with a constant stepsize independent of the number of functions. Such non-diminishing stepsizes allows them to be fast.
\end{abstract}

\section{Introduction}
\label{sc:intro}

In the past decade, first-order methods like
the proximal-gradient method and ADMM
have enjoyed wide popularity
due to their broad applicability, simplicity, and good empirical performance on problems with large data sizes.
However, there are many optimization problems such existing simple first-order methods cannot directly handle.
Without a simple and scalable method to solve them such optimization problems
have been excluded from machine learning and statistical modeling.
In this paper we present 
the proximal-proximal-gradient method
\eqref{eq:main-method},
a novel method
that expands the class of problems
that one can solve with a simple and scalable first-order method.

Consider the optimization problem
\begin{equation}
\mbox{minimize}
\quad
r(x)+\frac{1}{n}\sum_{i=1}^n(f_i(x)+g_i(x)),
\label{eq:main-prob}
\end{equation}
where $x\in \reals^d$ is the optimization variable,
$f_1,\dots,f_n$, $g_1,\dots,g_n$, and $r$
are convex, closed, and proper functions from
$\reals^d$ to $\reals\cup\{\infty\}$.
Furthermore, assume $f_1,\dots,f_n$ are differentiable.
We call the method
\begin{align}
x^{k+1/2}&= \prox_{\alpha r}\left(\frac{1}{n}\sum^n_{i=1}z_i^k\right)
\nonumber
\\
x_i^{k+1}&=
\prox_{\alpha g_i}\left(2x^{k+1/2}-z_i^k-\alpha \nabla f_i(x^{k+1/2})\right)
\nonumber
\\
z_i^{k+1}&=z_i^k+x_i^{k+1}-x^{k+1/2},
\tag{PPG}
\label{eq:main-method}
\end{align}
the \emph{proximal-proximal-gradient method}
\eqref{eq:main-method}.
The $x_i^{k+1}$ and $z_i^{k+1}$ updates are performed for all $i=1,\dots,n$
and $\alpha >0$ is a stepsize parameter.
To clarify,
$x$, $x_1,\ldots,x_n$ and $z_1,\ldots,z_n$ are all vectors in $\mathbb{R}^d$ ($x_i$ is not a component of $x$),
$x_1^{k+1},\dots,x_n^{k+1}$ and $x^{k+1/2}$
approximates the solution to Problem~\eqref{eq:main-prob}.

Throughout this paper we write $\prox_h$ for the \emph{proximal operator}
with respect to the function $h$, defined as
\[
\prox_h(x_0)=\argmin_x
\left\{
h(x)+\frac{1}{2}\|x-x_0\|_2^2
\right\}
\]
for a function $h: \reals^d\rightarrow\reals\cup\{\infty\}$.
When $h$ is the zero function, $\prox_h$ is the identity operator. When $h$ is convex, closed, and proper,
the minimizer that defines $\prox_h$
exists and is unique \cite{minty1962}.
For many interesting functions $h$,
the proximal operator $\prox_h$
has a closed or semi-closed form solution
and is computationally easy to evaluate
\cite{combettes2011, parikh2014}.
We loosely say such functions are
\emph{proximable}.

In general,
the proximal-gradient method or ADMM cannot directly solve
optimization problems expressed in the form of \eqref{eq:main-prob}.
When $f_1,\dots,f_n$ are not proximable,
ADMM either doesn't apply or must run another optimization
algorithm to evaluate the proximal operators at each iteration.
When $n\ge 2$ and  $g_1,\ldots,g_n$ are nondifferentiable nonseparable, so $g_1+\dots+g_n$ is not proximable
(although each individual $g_1,\dots,g_n$ is proximable).
Hence, proximal-gradient doesn't apply.

One possible approach to solving \eqref{eq:main-prob}
is to smooth the non-smooth parts and applying a (stochastic) gradient method.
Sometimes, however, keeping non-smooth part is essential.
For example, it is the non-smoothness of total variation 
penalty
that induces sharp edges in image processing. 
In these situations
\eqref{eq:main-method} is particularly useful as it can handle a large sum of smooth and non-smooth terms directly without smoothing.

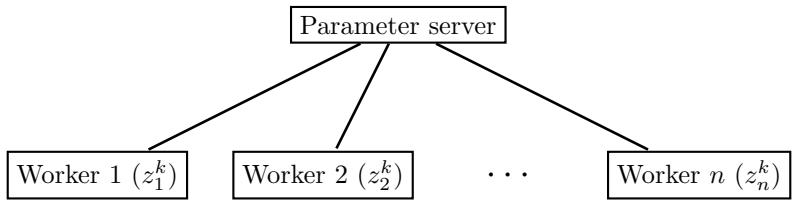
\begin{figure}
\begin{center}
\begin{tikzpicture}
[->,>=stealth',shorten >=1pt,auto,node distance=3cm,
                    thick,main node/.style={draw}]

  \node[main node] (1) at (4, 2) {Parameter server};
  \node[main node] (2) at (0, 0)  {Worker $1$ ($z_1^{k}$)};
  \node[main node] (3) at (3, 0)  {Worker $2$ ($z_2^{k}$)};
  \node[] (5) at (5.5, 0)  {\Large{$\cdots$}};
  \node[main node] (4) at (8, 0)  {Worker $n$ ($z_n^{k}$)};

  \path[line width=1pt,-]
    (1) edge node [left] {} (2)
    (1) edge node [left] {} (3)
    (1) edge node [left] {} (4);
\end{tikzpicture}
\end{center}
\caption{
When \eqref{eq:main-method} is implemented on a parameter server computing model,
the worker nodes communicate (synchronously) with the parameter server
but not directly with each other.
}
\label{fig:paramterserver}
\end{figure}

\vspace{0.1in}\noindent\textbf{Distributed PPG.}
To understand the algorithmic structure of the method,
it is helpful to see how \eqref{eq:main-method}
is well-suited for a distributed computing network.
See Figure~\ref{fig:paramterserver},
which illustrates a parameter server computing model
with a master node and $n$ worker nodes.

At each iteration, the workers
send their $z^{k}_i$ to the parameter server,
the parameter server computes $x^{k+1/2}$
and broadcasts it to the workers,
and each worker $i$ computes $z^{k+1}_i$ 
with access to $f_i$, $g_i$ and $z^k_i$
for $i=1,\dots,n$.
The workers maintain their private copy of $z^k_i$
and do not directly communicate with each other.

In other words, a distributed implementation 
performs step 1 of \eqref{eq:main-method},
the $x^{k+1/2}$ update, with an all-reduce operation.
It performs step 2 and step 3, the $x^{k+1}_i$ and $z^{k+1}_i$ updates,
in parallel.

\vspace{0.1in}\noindent\textbf{Time complexity, space complexity, and parallelization.}
At each iteration, \eqref{eq:main-method} evaluates
the proximal operators with respect to
$r$ and $g_1,\dots,g_n$ and computes the gradients of $f_1,\dots,f_n$.
So based on the iteration cost alone, we can predict
\eqref{eq:main-method} to be useful
when $r$ and $g_1,\dots, g_n$ are individually proximable
and the gradients of $f_1,\dots,f_n$ are easy to evaluate.
If, furthermore,
the number of iterations required to reach necessary accuracy
is not exorbitant,
\eqref{eq:main-method} is actually useful.

Let's say the computational costs of evaluating $\prox_{\alpha r}$ is $c_r$,
$\prox_{\alpha g_i}$ is $c_g$ for $i=1,\dots,n$,
and $\nabla f_i$ is $c_f$ for $i=1,\dots,n$.
Then the time complexity of \eqref{eq:main-method}
is $\mathcal{O}(nd+c_r+nc_g+nc_f)$ per iteration (recall $x\in\reals^d$).
As discussed, this cost can be reduced with parallelization.
Computing $x^{k+1/2}$ involves computing an average (a reduce operation), computing $\prox_{\alpha r}$, and a broadcast,
and computing $z^{k+1}_i$ for $i=1,\dots,n$ is embarrassingly parallel.
The space complexity of \eqref{eq:main-method} is $\mathcal{O}(nd)$
since it must store $z_1^k,\dots z_n^k$.
($x_1^{k},\dots,x_n^k$ need not be stored.)


When the problem has sparse structure,
the computation time and storage can be further reduced.
For example, if $f_i+g_i$ does not depend on $(x_i)_1$ for some $i$,
then $(z_i)_1$ and $(x_i)_1$ can be eliminated from the algorithm
since $(z_i)^{k+1}_1 = (x^{k+1/2})_1$.

The storage requirement of $O(nd)$ is fundamentally
difficult to improve upon due to the $n$ non-differentiable terms.
Consider the case where $r=f_1=\dots=f_n=0$:
 \[
\begin{array}{ll}
\mbox{minimize}&
g_1(x)+\cdots +g_n(x).
\end{array}
 \]
If $g_1,\dots,g_n$ were differentiable, then $\nabla g_1(x^*)+\dots+\nabla g_n(x^*)=0$
would certify $x^*$ is optimal.
However, we allow $g_1,\dots,g_n$ to be non-differentiable, so one must find
a particular set of subgradients $u_i\in \partial g_i(x^*)$ for $i=1,\dots,n$
such that $u_1+\dots+ u_n=0$ to certify $x^*$ is optimal.
The choices of subgradients,
$u_1,\dots,u_n$, depend on each other and cannot be found independently.
 In other words, certifying optimality requires $\mathcal{O}(nd)$ information, and that is what PPG uses.
For comparison, ADMM also uses $\mathcal{O}(nd)$ storage when used to minimize
a sum of $n$ non-smooth functions.

\vspace{0.1in}\noindent\textbf{Stochastic PPG.}
Each iteration of \eqref{eq:main-method} updates $z^{k+1}_i$
for all $i=1,\dots,n$,
which takes at least $\mathcal{O}(nd)$ time per iteration.
In Section~\ref{sc:sppg} we present the method \eqref{eq:main-sppg}
which can be considered a stochastic variation of \eqref{eq:main-method}.
Each iteration of \eqref{eq:main-sppg} only updates
one $z^{k+1}_i$ for some $i$ and therefore can take
as little as $\mathcal{O}(d)$ time per iteration.
When compared epoch by epoch,
\eqref{eq:main-sppg} can be faster than \eqref{eq:main-method}.

\vspace{0.1in}\noindent\textbf{Convergence.}
Assume Problem~\eqref{eq:main-prob} has a solution (not necessarily unique)
and meets a certain regularity condition.
Furthermore, assume each $f_i$ in
has $L$-Lipschitz gradient for $i=1,\dots,n$,
so $\|\nabla f_i(x)-\nabla f_i(y)\|_2\le L\|x-y\|_2$ for all
$x,y\in\reals^d$ and $i=1,\dots,n$.
Then \eqref{eq:main-method} converges
to a solution of
Problem~\eqref{eq:main-prob}
for $0<\alpha<3/(2L)$.
In particular, we do not need strong convexity
to establish convegence.

In section~\ref{sc:conv} we discuss
convergence in more detail.
In particular, we prove both that the iterates converge
to a solution and that the objective values converge to
the optimal value.




\vspace{0.1in}\noindent\textbf{Contribution of this work.}
The methods of this paper, \eqref{eq:main-method} and \eqref{eq:main-sppg},
are novel methods that can directly handle a sum of many 
differentiable and non-differentiable but proximable functions.
To solve such problems, exising first-order methods like ADMM
must evaluate proximal operators of differentiable functions,
which may hinder computational performance if said operator has no closed form solution.
Furthermore, the simplicity of our methods allows simple and efficient parallelization,
a point we discuss briefly.

The theoretical analysis of  \eqref{eq:main-method} and \eqref{eq:main-sppg},
especially that of \eqref{eq:main-sppg}, is novel.
As we discuss later, \eqref{eq:main-sppg} can be interpreted as a generalization to
varianced reduced gradient methods like Finito/MISO or SAGA.
The techniques we use to analyze \eqref{eq:main-sppg}
is different from those used to analyze
other varianced reduced gradient methods,
and we show more general (albeit not faster) convergence guarantees.
In particular, we establish almost sure convergence of iterates, and we do so without
any strong convexity assumptions. To the best of our knowledge, the existing
varianced reduced gradient method literature does not prove such results.

Finally, our method is the first work to establish a clear connection between
operator splitting methods and varianced reduced gradient methods.
As the name implies, existing varianced reduced gradient methods
view the method as improving, by reducing variance,
stochastic gradient methdos.
Our work identifies Finito/MISO as stochastic block-coordinate update applied to an operator splitting method.
It is this observation that allows us to analyze \eqref{eq:main-sppg},
which generalizes Finito/MISO to handle a sum of non-differentiable but proximable functions.

\section{Relationship to other methods}
In this section, we discuss how \eqref{eq:main-method} generalizes certain known methods.
To clarify, PPG is a proper generalization of these existing methods
and cannot be analyzed as a special case of one of these methods.

\vspace{0.1in}\noindent\textbf{Proximal-gradient.}
When $g_i=0$ for $i=1,\dots,n$ in \eqref{eq:main-prob},
\eqref{eq:main-method} simplifies to
\begin{align*}
x^{k+1}&= \prox_{\alpha r}\left(
x^{k}-
\frac{\alpha }{n}\sum^n_{i=1}\nabla f_i(x^{k})\right).
\end{align*}
This method is the called proximal-gradient method or forward-backward splitting
\cite{passty1979, combettes2005}. 


\vspace{0.1in}\noindent\textbf{ADMM.}
When $f=r=0$ in \eqref{eq:main-prob},
\eqref{eq:main-method} simplifies to
\begin{align*}
x^{k+1}_i&=
\argmin_x
\left\{
g_i(x)-(y^k_i)^Tx+\frac{1}{2\alpha}\|x-x^k\|_2^2
\right\}\\
y^{k+1}_i&=y^k_i+
\alpha(\bar{x}^{k+1}-x^{k+1}_i)
\end{align*}
where
\[
\bar{x}^{k+1}=\frac{1}{n}\sum^n_{i=1}x^{k+1}_i.
\]
This is also an instance of ADMM \cite{glowinski1975,gabay1976}.
See \S7.1 of \cite{boyd2011}. 


\vspace{0.1in}\noindent\textbf{Generalized forward-backward splitting.}
When $r=0$ and $f_1=f_2=\dots=f_n=f$ in \eqref{eq:main-prob},
\eqref{eq:main-method} simplifies to
\begin{align*}
z_i^{k+1}&=z_i^k-x^{k}+
\prox_{\alpha g_i}\left(2x^{k}-z_i^k-\alpha \nabla f(x^{k})\right)\\
x^{k+1}&=\frac{1}{n}\sum^n_{i=1}z_i^{k+1}.
\end{align*}
This is an instance of generalized forward-backward splitting
\cite{raguet2013}.

\vspace{0.1in}\noindent\textbf{Davis-Yin splitting.}
When $n=1$ in \eqref{eq:main-prob}, \eqref{eq:main-method} reduces to
Davis-Yin splitting,
and much of convergence analysis for \eqref{eq:main-method} is inspired by
that of Davis-Yin splitting \cite{davis2017}.
However, \eqref{eq:main-method} is more general and parallelizable.
Furthermore, \eqref{eq:main-prob} has a stochastic variation
\eqref{eq:main-sppg}.

\section{Stochastic PPG}
\label{sc:sppg}
Each iteration of \eqref{eq:main-method} updates $z_i^{k+1}$ for all $i=1,\dots,n$,
which requires at least
$\mathcal{O}(nd)$ time (without parallelization or sparse structure).
Often, the data that specifies Problem~\eqref{eq:main-prob}
is of size $\mathcal{O}(nd)$,
and, roughly speaking, \eqref{eq:main-method} must process the entire dataset every iteration.
This cost of $\mathcal{O}(nd)$ time per iteration may be inefficient in certain applications.


The following method, which we call the
\emph{stochastic proximal-proximal-gradient} method
\eqref{eq:main-sppg},
overcomes this issue:
\begin{align}
x^{k+1/2}&= \prox_{\alpha r}\left(\frac{1}{n}\sum^n_{i=1}z_i^k\right)\nonumber\\
i(k)&\sim \text{Uniform}(\{1,\dots,n\})\nonumber\\
x_{i(k)}^{k+1}&=
\prox_{\alpha g_{i(k)}}\left(2x^{k+1/2}-z_{i(k)}^k-\alpha \nabla f_{i(k)}(x^{k+1/2})\right)
\nonumber
\\
z_{j}^{k+1}&=\begin{cases}
z_{i(k)}^k+x_{i(k)}^{k+1}-x^{k+1/2}&\text{for }j=i(k),\\
z_{j}^k&\text{for }j\ne i(k).
\end{cases}
\label{eq:main-sppg}
\tag{S-PPG}
\end{align}
At each iteration, only $z^{k+1}_{i(k)}$ is updated,
where the index $i(k)$ is chosen uniformly at random from $\{1,\dots,n\}$.
We can interpret \eqref{eq:main-sppg}
as a stochastic or coordinate update version of \eqref{eq:main-method}.

\vspace{0.1in}\noindent\textbf{Time and space complexity.}
The space requirement of \eqref{eq:main-sppg}
is no different from that of \eqref{eq:main-method};
both methods use $\mathcal{O}(nd)$ space
to store $z^k_1,\dots, z^k_n$.
However, the cost per iteration of
\eqref{eq:main-sppg} can be as low as  $\mathcal{O}(d)$.
This is achieved with the following simple trick: maintain the quantity
\[
\bar{z}^k=\frac{1}{n}\sum^n_{i=1}z_i^k,
\]
and update it with
\[
\bar{z}^{k+1}=\bar{z}^k+(1/n)(x^{k+1}_{i(k)}-x^{k+1/2}).
\]




\vspace{0.1in}\noindent\textbf{Application to big-data problems.}
Consider a big-data problem setup where
the data that describes Problem~\eqref{eq:main-prob}
is stored on a hard drive,
but is too large to fit in a system's memory.
Under this setup, an optimization algorithm
that goes through the entire dataset
every iteration is likely impractical.

\eqref{eq:main-sppg} can handle this setup effectively
by keeping the data and the $z_i^k$ iterates on the hard drive
as illustrated in Figure~\ref{fig:sppg-big}.
At iteration $k$, \eqref{eq:main-sppg} selects index $i(k)$,
reads block $i(k)$ containing $f_{i(k)}$, $g_{i(k)}$, and $z_{i(k)}^k$ from the hard drive, 
performs computation, updates $\bar{z}^{k+1}$,
and writes $z_{i(k)}^{k+1}$ back to block $i(k)$.
The $\bar{z}^k$ iterate is used and updated every iteration and therefore should be
stored in memory.

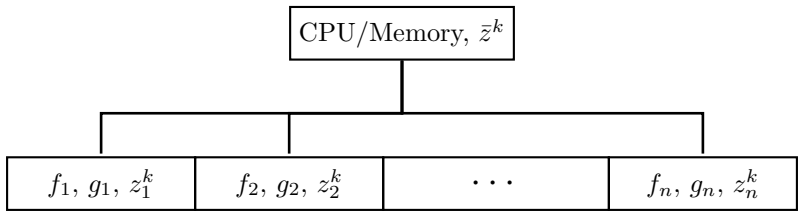
\begin{figure}
\begin{center}
\begin{tikzpicture}
[->,>=stealth',shorten >=1pt,auto,node distance=3cm,minimum width=2.5cm, minimum height=0.7cm,
                    thick,main node/.style={draw}]

  \node[main node] (1) at (4, 2) {CPU/Memory, $\bar{z}^k$};
  \node[main node] (2) at (0, 0)  {$f_1$, $g_1$, $z_1^k$};
  \node[main node] (3) at (2.5, 0)  {$f_2$, $g_2$, $z_2^k$};
  \node[main node, minimum width=3cm] (5) at (5.25, 0)  {\Large{$\cdots$}};
  \node[main node] (4) at (8, 0)  {$f_n$, $g_n$, $z_n^k$};

    \draw[line width=1pt,-] (1) -- node[node distance=1cm and 3.5cm,below=of 1] {}($(1.south)+(0,-0.7)$) -| (2) ;
    \draw[line width=1pt,-] (1) -- node[node distance=1cm and 3.5cm,below=of 1] {}($(1.south)+(0,-0.7)$) -| (3) ;
    \draw[line width=1pt,-] (1) -- node[node distance=1cm and 3.5cm,below=of 1] {}($(1.south)+(0,-0.7)$) -| (4) ;

\end{tikzpicture}
\caption{An illustration of \eqref{eq:main-sppg}
applied to a big-data problem. The bottom blocks
represent $n$ blocks of data stored on the hard drive.
The CPU accesses only one of the $n$ blocks per iteration.
}
\end{center}
\label{fig:sppg-big}
\end{figure}

\vspace{0.1in}\noindent\textbf{Interpretation as a variance reduced gradient method.}
Remarkably \eqref{eq:main-sppg}
converges to a solution with a fixed value of $\alpha$ (which is independent of $n$).
Many traditional and modern stochastic optimization methods require
their step sizes to diminish to $0$, theoretically and empirically,
and this limits their rates of convergence.

On the other hand, several modern
``variance reduced gradient'' methods
take advantage of a finite sum structure
similar to that of \eqref{eq:main-prob}
and achieve a faster rate with a constant step size.
In fact, these methods achieve
superior performance compared to full gradient updates.
Such methods include Finito, MISO, SVRG, SAG, SAGA, and SDCA
\cite{leroux2012,mairal2013,johnson2013,zhang2013, shalev2103,defazio2014b, defazio2014,nitanda2014,xiao2014, mairal2015,lan2015,schmidt2016, shalev2016, shalev2016b}.

In fact, \eqref{eq:main-sppg} directly generalizes Finito and MISO.
When $g_1=\dots=g_n=0$ and $r=0$ in \eqref{eq:main-prob},
we can rewrite \eqref{eq:main-sppg} as
\begin{align*}
w^k&=\frac{1}{n}\sum^n_{i=1}z_i^k\\
i(k)&\sim \text{Uniform}(\{1,\dots,n\})\nonumber\\
\phi^{k+1}_{i(k)}&=w^k\\
z^{k+1}_{i(k)}
&=
\phi^{k+1}_{i(k)}-\alpha\nabla f_{i(k)}(\phi^{k+1}_{i(k)})\\
 z^{k+1}_{j}&=z^{k}_{j}\quad\text{for }j\ne i(k),
\end{align*}
which is Finito and an instance of MISO \cite{mairal2013,defazio2014}.
Therefore it is appropriate to view \eqref{eq:main-sppg}
as a variance reduced gradient method as opposed to a stochastic gradient method.
Of course, \eqref{eq:main-sppg} is more general
as it can handle a sum of non-smooth terms as well.

\vspace{0.1in}\noindent\textbf{Comparison to SAGA.}
\eqref{eq:main-sppg} is more general than 
SAGA \cite{defazio2014b} 
as it can directly handle a sum of many non-differentiable functions.
However, when $g_1=\dots=g_n=0$ in \eqref{eq:main-prob},
one can use SAGA,
and it is interesting to compare 
\eqref{eq:main-sppg} with SAGA under this scenario.

The difference in storage requirement is small.
\eqref{eq:main-sppg}
must store at least $nd$ numbers 
while SAGA 
must store at least 
$(n+1)d$.
This is because 
SAGA stores the current iterate and $n$ gradients,
while \eqref{eq:main-sppg} only stores $z_1,\dots,z_n$.

On the other hand, there is a
difference in memory access.
At each iteration 
\eqref{eq:main-sppg}
reads and updates $\bar{z}$
while SAGA reads and updates
the current iterate and average gradient.
So \eqref{eq:main-sppg}
reads $n$ fewer numbers
and writes $n$ fewer numbers
per iteration compared to SAGA.
Both SAGA and \eqref{eq:main-sppg}
read and update information corresponding
to a randomly chosen index,
and the memory access for this is  comparable.

\vspace{0.1in}\noindent\textbf{Comparison to stochastic proximal iteration.}
When $f_1=\dots=f_n=0$ and $r=0$ in \eqref{eq:main-prob},
the optimization problem  is
\[
\mbox{minimize}
\quad
\frac{1}{n}
\sum^n_{i=1}g_i(x).
\]
A stochastic method that can solve this problem is
stochastic proximal iteration:
\begin{align*}
i(k)&\sim \text{Uniform}(\{1,\dots,n\})\nonumber\\
x^{k+1}&=\prox_{\alpha^k g_{i(k)}}(x_k),
\end{align*}
where $\alpha^k$ is a appropriately decreasing step size.
Stochastic proximal iteration has been studied under many names
such as stochastic proximal point,
incremental stochastic gradient, and implicit stochastic gradient
\cite{langford2009,kulis2010, mcmahan2011, bertsekas2011, toulis2014, ryu2014,  wang2015,  bianchi2016, salim2016, wang2016,toulis2017}.

Stochastic proximal iteration requires the step size $\alpha^k$ to be diminishing,
whereas \eqref{eq:main-sppg} converges with a constant step size.
As mentioned,
optimization methods with diminishing step size tend to have slower rates,
which we can observe in the numerical experiments.
We experimentally compare \eqref{eq:main-method} and \eqref{eq:main-sppg}
to stochastic proximal iteration in Section~\ref{sc:experiments}.


\vspace{0.1in}\noindent\textbf{Communication efficient implementation.}
One way to implement \eqref{eq:main-sppg} 
on a distributed computing network 
so that communication between nodes are minimized 
is to have nodes update and randomly pass around the $\overline{z}$ variable.
See Figure~\ref{fig:comm}.
Each iteration, the current node updates $\overline{z}$
and passes it to another randomly selected node.
Every neighbor and the current node is chosen with probability $1/n$.

The communication cost of this implementation of \eqref{eq:main-sppg} 
is $\mathcal{O}(d)$ per iteration. When the number of iterations required for convergence
is not large, this method is communication efficient.
For recent work on communication efficient optimization methods, see
\cite{zhang2012,mota2013,yang2013, jaggi2014,shamir2014,arjevani2015, zhang20152}.

\begin{figure}
\begin{center}
\begin{tikzpicture}[->,>=stealth',shorten >=1pt,auto,node distance=4cm,
                thick,main node/.style={draw,font=\Large\bfseries}]
\begin{scope}[every node/.style={rounded rectangle,thick,draw, align=center}]
    \node[line width=0.4mm]  (A) at (0,0) {$f_3$, $g_3$, $z_3$ \\ current $\overline{z}$};
    \node (B) at (0,2*1) {$f_4$, $g_4$, $z_4$ \\ past $\overline{z}$};
    \node  (C) at (2*0.866,2*1.5) {$f_5$, $g_5$, $z_5$ \\ past $\overline{z}$};
    \node (D) at (2*2*0.866,2*1) {$f_6$, $g_6$, $z_6$ \\ past $\overline{z}$};
    \node (E) at (2*2*0.866,2*0) {$f_1$, $g_1$, $z_1$ \\ past $\overline{z}$};
    \node (F) at (2*0.866,-2*0.5) {$f_2$, $g_2$, $z_2$ \\ past $\overline{z}$};
\end{scope}
  \path (A) edge  (B)
   (A) edge (C)
     (A) edge  (D)
 (A) edge  (E)
(A) edge (F);
  \path [dashed, -, line width=0.1mm](C) edge  (B);
  \path [dashed, -, line width=0.1mm](D) edge  (B);
    \path [dashed, -, line width=0.1mm](E) edge  (B);
      \path [dashed, -, line width=0.1mm](F) edge  (B);
    \path [dashed, -, line width=0.1mm](C) edge  (D);
        \path [dashed, -, line width=0.1mm](C) edge  (E);
            \path [dashed, -, line width=0.1mm](C) edge  (F);
    \path [dashed, -, line width=0.1mm](E) edge  (D);
    \path [dashed, -, line width=0.1mm](F) edge  (D);
    \path [dashed, -, line width=0.1mm](E) edge  (F);
\draw
 (A) to [->,out=180,in=270,looseness=5, below] (A);
\end{tikzpicture}
\end{center}
\caption{
Distributed implementation of \eqref{eq:main-sppg}
without synchronization.
Node $3$ has the current copy of $\overline{z}$,
and will pass it to another randomly selected node.
}
\label{fig:comm}
\end{figure}
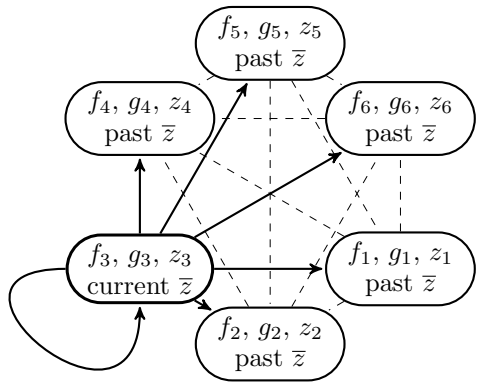

\vspace{0.1in}\noindent\textbf{Convergence.}
Assume Problem~\eqref{eq:main-prob} has a solution (not necessarily unique)
and meets a certain regularity condition.
Furthermore, assume each $f_i$ in
has $L$-Lipschitz continuous gradient for $i=1,\dots,n$,
so $\|\nabla f_i(x)-\nabla f(y)\|\le L\|x-y\|$ for all
$x,y\in\reals^d$ and $i=1,\dots,n$.
Then \eqref{eq:main-sppg} converges
to a solution of
Problem~\eqref{eq:main-prob}
for $0<\alpha<3/(2L)$.
In particular, we do not assume strong convexity to establish convergence,
whereas many of the mentioned variance reduced gradient methods do.
\section{Convergence}
\label{sc:conv}
In this section, we present and discuss
the convergence of \eqref{eq:main-method} and \eqref{eq:main-sppg}.

For this section, we introduce some new notation.
We write $\vx=(x_1,\dots,x_n)$,
and we use other boldface letters like $\boldsymbol{\nu}$ and $\vz$
in a similar manner.
We use the bar notation for $\bar{z}=(z_1+\dots+z_n)/n$.
We write
\begin{align*}
\bar{f}(x)&=(f_1(x)+\dots+f_n(x))/n\\
\bar{g}(x)&=(g_1(x)+\dots+g_n(x))/n,
\end{align*}
and, with some abuse of notation, we write
\begin{align*}
\bar{f}(\vx)=(f_1(x_1)+\dots+f_n(x_n))/n\\
\bar{g}(\vx)=(g_1(x_1)+\dots+g_n(x_n))/n.
\end{align*}
Note that $f_i$ and $g_i$ for depend $x_i$ instead of a common $x$.
The main problem \eqref{eq:main-prob} is equivalent to
\begin{align}
&\mbox{minimize}~~r(x) + \bar{f}(\vx) + \bar{g}(\vx)\\
&\mbox{subject to}~~ x - x_i = 0,\quad i=1,\ldots,n,
\end{align}
where $x$ and $\vx=(x_1,\ldots,x_n)$ are the optimization variables.
Convex duality tells us that
under certain regularity conditions $x^\star$ is a solution
of Problem~\eqref{eq:main-prob} if and only if
$(x^\star,\vx^\star,\boldsymbol{\nu}^\star)$ is a saddle point of the Lagrangian
\begin{equation}
L(x,\vx,\boldsymbol{\nu})
=
r(x)+
\bar{f}(\vx)+\bar{g}(\vx)+
\frac{1}{n}\sum_{i=1}^n\nu_i(x_i-x),
\label{eq:lagrangian}
\end{equation}
where $\boldsymbol{\nu}^\star=(\nu_1^\star,\dots,\nu_n^\star)$ is a dual solution,
and $\vx^\star=(x^\star,\dots,x^\star)$.
We simply assume the Lagrangian \eqref{eq:lagrangian}
has a saddle point.
This is not a stringent requirement and is merely assumed to avoid pathologies.

Define the mapping $p(\vz)$ as
\begin{align}
(p(\vz))_i&=(1/\alpha)(x-x_i') \quad \text{for }i=1,\dots,n,\nonumber\\
\mbox{where}\quad x & = \prox_{\alpha r}(\bar{z}),\label{eq:p-mapping}
\\
x_i' & = \prox_{\alpha g_i}\left(2x-z_i-\alpha \nabla f_i(x)\right).\nonumber
\end{align}
With this notation, we can express \eqref{eq:main-method} as
\[
\vz^{k+1} = \vz^k - \alpha p(\vz^k),
\]
and we can say $\vz$ is a fixed point of
\eqref{eq:main-method} and \eqref{eq:main-sppg}
if and only if $p(\vz)=0$.

\begin{lemma}
\label{lm:encoding}
$\vz^\star = (z_1^\star,\ldots,z_n^\star)$
is a fixed point of \eqref{eq:main-method} and \eqref{eq:main-sppg}
if and only if
$z_i^\star=x^\star+\alpha \nu_i^\star$
for $i=1,\dots,n$,
where $x^\star$ and $\nu_1^\star,\dots,\nu_n^\star$
are primal and dual solutions.
In particular, we can recover the $x^\star$ as
$x^\star=\prox_{\alpha r}((1/n)(z^\star_1+\dots+z^\star_n))$.
\end{lemma}
This lemma provides us insight as to why the method converges.
Let's write $z_i^k=x_i^k+\alpha \nu_i^k$. Then
the updates can be written as
\[
x^{k+1/2}
=\argmin_x\left\{
r(x)+(\bar\nu_i^k)^T(\bar{x}_i^k-x)+\frac{1}{2\alpha}\|x-\bar{x}_i^k\|_2^2
\right\}
\]
and
\begin{align*}
x^{k+1}_i
&=\argmin_x\Big\{
f(x^{k+1/2})+(\nabla f(x^{k+1/2}))^T(x-x^{k+1/2})\\
&+g(x)+(x-2x^{k+1/2}+x_i^k)^T\nu_i^k+\frac{1}{2\alpha}\|x-2x^{k+1/2}+x_i^k\|_2^2
\Big\}.
\end{align*}
Since
$(x^\star,\vx^\star,\boldsymbol{\nu}^\star)$ is a saddle point of the Lagrangian,
we can see that $z^\star_i=x^\star+\alpha\nu_i^\star$ is a fixed point of \eqref{eq:main-method}.

\subsection{Deterministic analysis.}
We examine the convergence results for \eqref{eq:main-method}.
\begin{theorem}
\label{thm:seq}
Assume $f_1,\dots,f_n$ are differentiable and have $L$-Lipschitz continuous gradients.
Assume the Lagrangian \eqref{eq:lagrangian}
has a saddle point
and $0<\alpha<3/(2L)$.
Then the sequence
$\|p(\vz^k)\|_2\rightarrow 0$ monotonically
with rate
\[
\|p(\vz^k)\|_2\le
\mathcal{O}(1/\sqrt{k}).
\]
Furthermore $\vz^k\rightarrow \vz^\star$,
$x^{k+1/2}\rightarrow x^\star$, and
$x^{k+1}_i\rightarrow x^\star$ for all $i=1,\dots,n$, where $\vz^\star$ is a fixed point of \eqref{eq:main-method} and $x^\star$ is a solution of \eqref{eq:main-prob}.
\end{theorem}
Theorem~\ref{thm:seq}
should be understood as two related but separate results.
The first result states $p(\vz^k)\rightarrow 0$ and provides a rate.
Since $p(\vz)=0$ implies $\prox_{\alpha r}(\bar{z})$ is a solution,
the rate does quantify progress.
The second result states that the iterates of \eqref{eq:main-method} converge
but with no guarantee of rate (just like gradient descent without strong convexity).

To obtain a more direct measure of progress, define
\begin{align*}
E^k&=
r(x^{k+1/2})+\bar{f}(x^{k+1/2})+\bar{g}(\vx^{k+1})
-\big(r(x^\star)+\bar{f}(x^{\star})+\bar{g}(x^\star)\big).
\end{align*}
$E^k$ is almost like the suboptimality of iterates, but not quite,
as the point where $\bar{g}$ is evaluated at
is different from the point where $r$ and $\bar{f}$ is evaluated at.
In fact, $E^k$ is not  necessarily positive. Nevertheless,
we can show a rate on $|E^k|$.
\begin{theorem}
\label{thm:fun-rate}
Under the setting of Theorem \ref{thm:seq},
\[
|E^k|\le
\mathcal{O}(1/\sqrt{k}).
\]
\end{theorem}
Define a similar quantity
\begin{align*}
e^k&=
\big(r(x^{k+1/2})+\bar{f}(x^{k+1/2})+\bar{g}(x^{k+1/2})\big)-\big(r(x^\star)+\bar{g}(x^\star)+\bar{f}(x^{\star})\big).
\end{align*}
While $e^k$ truly measures suboptimality of $x^{k+1/2}$,
it is possible for $e^k=\infty$ for all $k=1,2,\dots$ because $r$ and $g$ are possibly nonsmooth and valued $\infty$ at some points.
We need an additional assumption
for $e^k$ to be a meaningful quantity.
\begin{corollary}
\label{cor:fun-rate}
Assume the setting of Theorem \ref{thm:seq}.
Further assume $\bar{g}(x)$ is Lipschitz continuous with
parameter $L_g$. Then
\[
0\le e^k \le |E^k|+L_g\|p(\vz^k)\|
\]
and
\[
e^k=\mathcal{O}(1/\sqrt{k}).
\]
\end{corollary}
The proof of Corollary~\ref{cor:fun-rate} follows immediately
from combining Theorems~\ref{thm:seq} and \ref{thm:fun-rate}
with $e^k$'s and $L_g$'s definitions.

The $1/\sqrt{k}$ rates for Theorem~\ref{thm:fun-rate} and Corollary~\ref{cor:fun-rate}
can be improved to the $1/{k}$ rates by using the
\emph{ergodic iterates}:
\[
x^{k+1/2}_\mathrm{erg}=\frac{1}{k}\sum^k_{j=1}x^{j+1/2},
\quad
\vx^{k+1}_\mathrm{erg}=\frac{1}{k}\sum^k_{j=1}\vx^{j+1}.
\]
With these ergodic iterates, we define
\begin{align*}
E^k_\mathrm{erg}&=
\big(r(x^{k+1/2}_\mathrm{erg})+\bar{f}(x^{k+1/2}_\mathrm{erg})+\bar{g}(\vx^{k+1}_\mathrm{erg})\big)
-\big(r(x^\star)+\bar{g}(\vx^\star)+\bar{f}(x^{\star})\big),\\
e^k_\mathrm{erg}&=
\big(r(x^{k+1/2}_\mathrm{erg})+\bar{f}(x^{k+1/2}_\mathrm{erg})+\bar{g}(x^{k+1/2}_\mathrm{erg})\big)
-\big(r(x^\star)+\bar{g}(\vx^\star)+\bar{f}(x^{\star})\big).
\end{align*} 
\begin{theorem}
\label{thm:erg}
Assume the setting of Theorem \ref{thm:seq}. Then
\[
|E^k_\mathrm{erg}|\le \mathcal{O}(1/k)
\]
Further assume $\bar{g}(x)$ is Lipschitz continuous with
parameter $L_g$. Then
\[
e^k_\mathrm{erg}\le \mathcal{O}(1/k).
\]
\end{theorem}
Finally, under rather strong conditions on the problems, linear convergence can also be shown.
\begin{theorem}\label{thm:lincvg}
Assume the setting of Theorem \ref{thm:seq}.
Furthermore, assume
$\bar{g}$ is differentiable with Lipschitz continuous gradient.
If one (or more) of $r$, $\bar{g}$, or $\bar{f}$ is strongly convex,
then \eqref{eq:main-method} converges linearly in the sense that
\[
\|\vz^k-\vz^\star\|_2^2\le \mathcal{O}(e^{-Ck})
\]
for some $C>0$. Consequently, $|E^k|$ and $e^k$ also converge linearly. 
\end{theorem}

\subsection{Stochastic analysis.}
As it turns out, the condition that guarantees
\eqref{eq:main-sppg}
converges is the same as that of
\eqref{eq:main-method}. In particular, there is not step size reduction!
\begin{theorem}
\label{thm:stc-conv}
Apply the same assumptions in Theorem \ref{thm:seq}. That is, assume $f_1,\dots,f_n$ are differentiable and have $L$-Lipschitz continuous gradients, and
assume
the Lagrangian \eqref{eq:lagrangian}
has a saddle point
and $0<\alpha<3/(2L)$.
Then the sequence
$\|p(\vz^k)\|_2\rightarrow 0$ with probability one at the rate
\[
\min_{i=0,\dots,k}\EE\|p(\vz^i)\|_2^2\le\mathcal{O}(1/k).
\]
Furthermore $\vz^k\rightarrow \vz^\star$,
$x^{k+1/2}\rightarrow x^\star$, and
$x^{k+1}_i\rightarrow x^\star$ for all $i=1,\dots,n$
with probability one.
\end{theorem}

The expected objective rates of 
\eqref{eq:main-sppg} also match those of 
\eqref{eq:main-method}.

\begin{theorem}\label{thm:storates}
Under the same setting of Theorem \ref{thm:stc-conv}, we have
\begin{align}
\EE|E^k|\le
\mathcal{O}(1/\sqrt{k})\quad\text{and}\quad |E^k_\mathrm{erg}|\le \mathcal{O}(1/k).
\end{align}
Further assume $\bar{g}(x)$ is Lipschitz continuous with
parameter $L_g$. Then
\[
\EE e^k_\mathrm{erg}\le \mathcal{O}(1/k).
\]
\end{theorem}
Due to space limitation, we state without proof that, under the setting of Theorem \ref{thm:lincvg},  
\eqref{eq:main-sppg} yields linearly convergent $\EE \|\vz^k-\vz^\star\|_2$, $\EE |E^k|$, and $\EE |e^k|$.

\section{Applications of PPG}
\label{sc:applications}
To utilize \eqref{eq:main-method},
a given optimization problem often needs to be recast
into the form of \eqref{eq:main-prob}.
In this section, we show some techniques for this
while presenting some interesting applications.

All examples
presented in this section
are naturally posed as
\begin{equation}
\mbox{minimize}
\quad
r(x)+\frac{1}{n}\sum_{i=1}^nf_i(x)+
\frac{1}{m}\sum^m_{j=1}g_j(x),
\label{eq:natural-prob}
\end{equation}
where $n\ne m$.
There is more than one way to recast Problem~\eqref{eq:natural-prob}
into the form of Problem~\eqref{eq:main-prob}.


Among these options, the most symmetric one, loosely speaking, is
\[
\mbox{minimize}
\quad
r(x)+\frac{1}{mn}
\sum^n_{i=1}
\sum^m_{j=1}
\big(f_i(x)+g_j(x)\big),
\]
which leads to the method
\begin{align*}
x^{k+1/2}&= \prox_{\alpha r}\left(\frac{1}{mn}
\sum^n_{i=1}
\sum^m_{j=1}
z_{ij}^k\right)
\\
x_{ij}^{k+1}&=
\prox_{\alpha g_j}\left(2x^{k+1/2}-z_{ij}^k-\alpha  \nabla f_i(x^{k+1/2})\right)
\\
z_{ij}^{k+1}&=z_{ij}^k+x_{ij}^{k+1}-x^{k+1/2}.
\end{align*}
In general, the product $mn$ can be quite large, and if so,
this approach is likely impractical.
In many examples, however,
$mn$ is not large as since $n=1$ or $m$ is small.

Another option, feasible when neither $n$ nor $m$ is small, is
\[
\mbox{minimize}
\quad
r(x)+\frac{1}{m+n}
\left(\sum^n_{i=1}((m+n)/n)f_i(x)
+
\sum^m_{j=1}((m+n)/m)g_j(x)\right),
\]
which leads to the method
\begin{align*}
x^{k+1/2}&= \prox_{\alpha r}\left(\frac{1}{n+m}
\left(
\sum^n_{i=1}y_i^k
+
\sum^m_{j=1}z_j^k\right)\right)\\
y_i^{k+1}&=
x^{k+1/2}-\alpha ((m+n)/n)\nabla f_i(x^{k+1/2})
\\
z_{j}^{k+1}&=z_{j}^k-x^{k+1/2}+
\prox_{\alpha ((m+n)/m)g_j}\left(2x^{k+1/2}-z_{j}^k\right).
\end{align*}

\vspace{0.1in}\noindent\textbf{Overlapping group lasso.}
Let $\mathcal{G}$ be a collection of groups of indices.
So $G\subseteq \{1,2,\dots,d\}$
for each $G\in \mathcal{G}$.
The groups can be overlapping, i.e.,
$G_1\cap G_2\ne \emptyset$ is possible
for $G_1,G_2\in \mathcal{G}$ and $G_1\ne G_2$.

We let $x_G\in \reals^{|G|}$ denote a subvector corresponding to the indices
of $G\in \mathcal{G}$,
where $x\in \reals^d$ is the whole vector.
So the entries $x_i$ for $i\in G$ form the vector $x_G$.

The overlapping group lasso problem is
\[
\mbox{minimize}
\quad
\frac{1}{2}\|Ax-b\|_2^2
+
\lambda_1
\sum_{G\in \mathcal{G}}\|x_G\|_{2},
\]
where $x\in \reals^d$ is the optimization variable,
$A\in \reals^{m\times d}$ and $b\in \reals^{m}$ are problem data,
and $\lambda_1>0$ is a regularization parameter.
As it is, the regularizer (the second term)
is not proximable when the groups overlap.

Partition the collection of groups $\mathcal{G}$ into
$n$ non-overlapping collections.
So $\mathcal{G}$ is a disjoint union of
$\mathcal{G}_1, \dots, \mathcal{G}_n$,
and if $G_1,G_2\in \mathcal{G}_i$ and $G_1\ne G_2$
then $G_1\cap G_2=\emptyset$ for $i=1,\dots,n$.
With some abuse of notation, we write
\[
\mathcal{G}_i^\mathsf{c}
=
\{i\in\{1,\dots,d\}\,|\,
i\notin G\text{ for all }G\in \mathcal{G}_i
\}
\]
Now we recast the problem into the form of \eqref{eq:main-prob}
\[
\mbox{minimize}
\quad
\frac{1}{2}\|Ax-b\|_2^2
+
\frac{1}{n}
\sum_{i=1}^n
\left(\sum_{G\in \mathcal{G}_i}\lambda_2 \|x_G\|_2\right),
\]
where $\lambda_2=n\lambda_1$.
The regularizer (the second term) is now a sum of $n$ proximable terms.

For example, we can have a setup with $d=42$, $n=3$,
$\mathcal{G}=\mathcal{G}_1\cup \mathcal{G}_2\cup \mathcal{G}_3$,
and
\begin{align*}
\mathcal{G}_1&= \{\{1,\dots,9\},\{10,\dots,18\},\{19,\dots,27\},\{28,\dots,36\}\}\\
\mathcal{G}_2& = \{\{4,\dots,12\},\{13,\dots,21\},\{22,\dots,30\},\{31,\dots,39\}\}\\
\mathcal{G}_3& = \{\{7,\dots,15\},\{16,\dots,24\},\{25,\dots,33\},\{34,\dots,42\}\}
\end{align*}
The groups within $\mathcal{G}_i$ do not overlap for each $i=1,2,3$,
and
$\mathcal{G}_2^\mathsf{c}=\{1,2,3,40,41,42\}$.

We view the first term as the $r$ term, the second term as the
sum of $g_1,\dots,g_n$, and $f_1=\dots=f_n=0$
in the notation of \eqref{eq:main-prob},
and apply \eqref{eq:main-method}:
\begin{align*}
x^{k+1/2}&=
(I+\alpha A^TA)^{-1}
\left(\alpha A^Tb+
\frac{1}{n}\sum_{i=1}^nz_i^k
\right)\\
z_i^{k+1/2}&=2x^{k+1/2}-z_i^k\\
(x_i^{k+1})_G&=
u_{\alpha\lambda_2 }
\left((z_i^{k+1/2})_G\right)
\quad
\text{for }
G\in \mathcal{G}_i
\\
(x_i^{k+1})_j&=
(z_i^{k+1/2})_j
\quad
\text{for }
j\in \mathcal{G}_i^\mathsf{c}
\\
z_i^{k+1}&=z_i^k+x_i^{k+1}-x^{k+1/2},
\end{align*}
where the indices $i$ implicitly run through $i=1,\dots,n$,
and $u_\alpha$, defined in the Section~\ref{s:theory},
is the vector soft-thresholding operator.

To reduce the cost of computing $x^{k+1/2}$,
we can precompute and store the Cholesky factorization
of the positive definite matrix $I+\alpha A^TA$
(which costs $\mathcal{O}(m^2d+d^3)$)
and the matrix-vector product $A^Tb$
(which costs $\mathcal{O}(md)$).
This cost is paid upfront once,
and the subsequent iterations
can be done in $\mathcal{O}(d^2+dn)$ time.

For recent work on overlapping group lasso, see
\cite{yuan2006,zhao2009,kim2010,mairal2010,jun2010,yuan2011, jenatton2011,chen2012,bien2013, yang2014}.

\vspace{0.1in}\noindent\textbf{Low-rank and sparse matrix completion.}
Consider the setup where we partially observe
a matrix $M\in \reals^{d_1\times d_2}$
on the set of indices $\Omega$.
More precisely, we observe $M_{ij}$ for $(i,j)\in \Omega$
while $M_{ij}$ for $(i,j)\notin \Omega$ are unknown.
We assume $M$ has a low-rank plus sparse structure,
i.e., $M=L^\mathrm{true}+S^\mathrm{true}$
with $L^\mathrm{true}$ is low-rank and $S^\mathrm{true}$ is sparse.
Here $L^\mathrm{true}$ models the true underlying structure
while $S^\mathrm{true}$ models outliers.
Furthermore, let's assume $0\le M_{ij}\le 1$ for all $(i,j)$.
The goal is to estimate
the unobserved entries of $M$,
i.e., $M_{ij}$ for $(i,j)\notin \Omega$.

To estimate $M$, we solve the following regularized regression
\begin{align*}
&\mbox{minimize}
\,\,
\lambda_1\|L\|_*
+
\lambda_2\|S\|_1
+
\sum_{(i,j)\in \Omega}
\ell(S_{ij}+L_{ij}-M_{ij})\\
&\mbox{subject to}\,\,
0\le S+L\le 1,
\end{align*}
where $S,L\in \reals^{d_1\times d_2}$ are the optimization variables,
the constraint $0\le S+L\le 1$ applies element-wise,
and $\lambda_1,\lambda_2>0$ are regularization parameters.
The constraint is proximable by Lemma~\ref{lm:prox-lemma},
and we can use \eqref{eq:main-method} either when $\ell$ is differentiable or when $\ell + I_{[0,1]}$ is proximable.

We view $n=1$, the first term as the $r$ term, the second term as the $g_1$ term,
and the last term as the $f_1$ term
in the notation of \eqref{eq:main-prob},
and apply \eqref{eq:main-method}:
\begin{align*}
L^{k+1/2}&=
t_\alpha\left(Z^{k}\right)\\
S^{k+1/2}_{ij}&=s_\alpha\left(
Y^{k}_{ij}\right)\quad\text{ for all }(i,j)\\
Z^{k+1/2}_{ij}&=
\left\{
\begin{array}{ll}
2L^{k+1/2}_{ij}-Z^k_{ij}&\text{ for }(i,j)\notin \Omega\\
2L^{k+1/2}_{ij}-Z^k_{ij}-\alpha (L^{k+1/2}_{ij}+S^{k+1/2}_{ij}-M_{ij})&\text{ for }(i,j)\in \Omega
\end{array}\right.\\
Y^{k+1/2}_{ij}&=
\left\{
\begin{array}{ll}
2S^{k+1/2}_{ij}-Y^k_{ij}&\text{ for }(i,j)\notin \Omega\\
2S^{k+1/2}_{ij}-Y^k_{ij}-\alpha (L^{k+1/2}_{ij}+S^{k+1/2}_{ij}-M_{ij})&\text{ for }(i,j)\in \Omega
\end{array}\right.
\\
A^{k+1}&=Z^{k+1/2}+Y^{k+1/2}\\
B^{k+1}&=Z^{k+1/2}-Y^{k+1/2}\\
L^{k+1}_{ij}&=
\frac{1}{2}\left(
\Pi_{[0,1]}(A^{k+1}_{ij})+B^{k+1}_{ij}\right)\quad\text{ for all }(i,j)\\
S^{k+1}_{ij}&=
\frac{1}{2}\left(
\Pi_{[0,1]}(A^{k+1}_{ij})-B^{k+1}_{ij}\right)\quad\text{ for all }(i,j)\\
Z^{k+1}&=Z^k+L^{k+1}-L^{k+1/2}\\
Y^{k+1}&=Y^k+S^{k+1}-S^{k+1/2}.
\end{align*}
$t_\alpha$ and $s_\alpha$,
defined in the Section~\ref{s:theory},
are respectively the matrix and scalar soft-thresholding operators.
$\Pi_{[0,1]}$ is the projection onto the interval $[0,1]$,
and the $L^{k+1}$ and $S^{k+1}$ updates follow from
Lemma~\ref{lm:prox-lemma}.
The only non-trivial operation for this method is computing the SVD
to evaluate $t_\alpha\left(Z^{k}\right)$.
All other operations are elementary and embarrassingly parallel.

For a discussion on low-rank + sparse factorization,
see \cite{chandrasekaran2011}.

\vspace{0.1in}\noindent\textbf{Regression with fused lasso.}
Consider the problem setup where
we have $Ax^\mathrm{true}=b$
and we observe $A$ and $b$.
Furthermore, the coordinates of $x^\mathrm{true}$
are ordered in a meaningful way
and we know a priori that
$|x_{i+1}-x_i|\le \varepsilon$
for $i=1,\dots,d-1$ and some $\varepsilon>0$.
finally, we also know that $x^\mathrm{true}$ is sparse.

To estimate $x$,
we solve the fused lasso problem
\begin{align*}
\mbox{minimize}&\quad
\lambda \|x\|_1+
\frac{1}{n}\sum^n_{i=1}
\ell_i(x)\\
\mbox{subject to}&\quad
|x_{i+1}-x_i|\le \varepsilon,\quad
i=1,\dots,d-1
\end{align*}
where
\[
\ell_i(x)=(1/2)(a^T_ix-y_i)^2
\]
and $x\in \reals^d$ is the optimization variable.

We recast the problem into the form of \eqref{eq:main-prob}
\begin{align*}
&\mbox{minimize}
\quad
\lambda
\|x\|_1\\
&\quad+
\frac{1}{2n}
\left(
\sum^n_{i=1}(\ell_i(x)+g_o(x))
+
\sum^n_{i=1}(\ell_i(x)+g_e(x))
\right)
\end{align*}
where
\begin{align*}
g_o(x)=
\sum_{i=1,3,5,\dots}
I_{[-\varepsilon,\varepsilon]}(x_{i+1}-x_i)\\
g_e(x)=
\sum_{i=2,4,6,\dots}I_{[-\varepsilon,\varepsilon]}(x_{i+1}-x_i)
\end{align*}
and\[
I_{[-\varepsilon,\varepsilon]}(x)=\left\{
\begin{array}{ll}
0&\text{if }|x|\le \varepsilon\\
\infty&\text{otherwise}.
\end{array}
\right.
\]
Since $g_o$ and $g_e$ are proximable
by Lemma~\ref{lm:prox-lemma},
we can apply \eqref{eq:main-method}.

For recent work on fused lasso, see
\cite{tibshirani2005, rapaport2008,tibshirani2008,hoefling2010,liu2010, nowak2011,ye2011,zhou2012}.

\vspace{0.1in}\noindent\textbf{Network lasso.}
Consider the problem setup
where we have an undirected graph $G=(E,V)$.
Each node $v\in V$ has a parameter to estimate $x_v\in \reals^d$
and an associated loss function $\ell_v$.
Furthermore, we know that neighbors of $G$ have similar parameters
in the sense that $\|x_u-x_v\|_2$ is small if $\{u,v\}\in E$
and that $x_v$ is sparse for each $v\in V$.

Under this model, we solve the
network lasso problem
\[
\mbox{minimize}
\quad\sum^{}_{v\in V}
\lambda_1\|x_v\|_1+
\ell_v(x_v)
+\sum_{\{u,v\}\in E}\lambda_2\|x_u-x_v\|_2,
\]
where $x_v$ for all $v\in V$ are the optimization variables,
and $\ell_v(x_v)$ for all $v\in V$ a differentiable loss function,
and $\lambda_1,\lambda_2>0$ are regularization parameters
\cite{hallac2015}.

Say the $G$ has an edge coloring $E_1,\dots,E_C$.
So $E_1,\dots ,E_C$ partitions $E$ such that
if $\{u,v\}\in E_c$ then $\{u,v'\}\notin E_c$
for any $v'\ne v$ and $c=1,\dots,C$.
Figure~\ref{fig:coloring} illustrates this definition.
(The chromatic index $\chi'(G)$ is the smallest possible value of $C$,
but $C$ need not be $\chi'(G)$.)
With the edge coloring, we recast the problem into the form of \eqref{eq:main-prob}
\begin{align*}
&\mbox{minimize}
\quad
\sum_{v\in V}\lambda_1\|x_v\|_1\\
&\quad +
\frac{1}{C}
\sum^C_{c=1}
\left(
\sum_{v\in V}\ell_v(x_v)
+\sum_{\{u,v\}\in E_c}\lambda_3\|x_u-x_v\|_2\right),
\end{align*}
where $\lambda_3=C\lambda_2$.

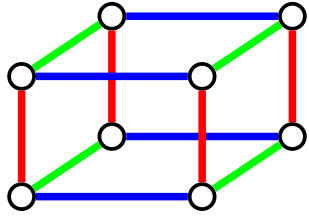
\begin{figure}
\begin{center}
    \begin{tikzpicture}[
    scale=.8,
            line width=1mm, 
        ]

        \node[state, line width=.5mm, minimum size=.2cm] (1) at (0,0) {};
        \node[state, line width=.5mm, minimum size=.2cm] (4) at (3,0)  {};
        \node[state,  line width=.5mm,minimum size=.2cm] (3) at (3,2)  {};
        \node[state,  line width=.5mm,minimum size=.2cm] (2) at (0,2) {};
        \node[state,  line width=.5mm,minimum size=.2cm] (7) at (1.5,3)  {};
        \node[state,  line width=.5mm,minimum size=.2cm] (5) at (4.5,1)  {};
        \node[state, line width=.5mm, minimum size=.2cm] (6) at (4.5,3)  {};
        \node[state,  line width=.5mm,minimum size=.2cm] (8) at (1.5,1)  {};

        \path[-, red] (1) edge (2);
        \path[-, red] (7) edge (8);
        \path[-, red] (5) edge (6);
        \path[-, blue] (2) edge  (3);
        \path[-, blue] (1) edge  (4);
        \path[-, blue] (7) edge  (6);
        \path[-, blue] (8) edge  (5);
        \path[-, green] (2) edge  (7);
        \path[-, green] (3) edge  (6);
        \path[-, green] (4) edge  (5);
        \path[-, green] (1) edge  (8);
        \path[-, red] (3) edge (4);
    \end{tikzpicture}
\caption{An edge coloring of the hypercube graph $Q_3$.}
\label{fig:coloring}
\end{center}
\end{figure}

We view the $\ell_1$ regularizer as the $r$ term,
the loss functions as the $f$ term,
and the summation over $E_c$ as the
$g$ terms
in the notation of \eqref{eq:main-prob},
and apply \eqref{eq:main-method}:
\begin{align*}
x^{k+1/2}_v&= s_{\alpha\lambda_1}\left(\frac{1}{C}\sum^C_{c=1}z_{cv}^k\right)\\
z_{cv}^{k+1/2}&=2x^{k+1/2}_v-z_{cv}^k-\alpha \nabla \ell_v(x^{k+1/2}_v)\\
s^k_c&=z_{cu}^{k+1/2}+z_{cv}^{k+1/2}\\
d^k_c&=z_{cu}^{k+1/2}-z_{cv}^{k+1/2}\\
x_{cu}^{k+1}&=s^k_c+u_{\alpha\lambda_3}\left(d^k_c\right)
\quad\text{for }\{u,v\}\in E_c\\
x_{cv}^{k+1}&=s^k_c-u_{\alpha\lambda_3}\left(d^k_c\right)
\quad\text{for }\{u,v\}\in E_c\\
x^{k+1}_{cv}&=z_{cv}^{k+1/2}
\quad\text{for }
\{v,u'\}\notin E_c \text{ for all }u'\in V\\
z_{cv}^{k+1}&=z_{cv}^k+x_{cv}^{k+1}-x^{k+1/2}_v,
\end{align*}
where the colors $c$ implicitly run through $c=1,\dots, C$
unless specified otherwise
and the nodes $v$ implicitly run through all $v\in V$.
Here
$s_\alpha$ and $u_\alpha$, defined in the Section~\ref{s:theory},
are respectively the scalar and vector soft-thresholding operators.

Although this algorithm, as stated, seemingly maintains $C$ copies of $x_v$
we can actually simplify it so that $v$
maintains $\min\{\deg(v)+1,C\}$ copies of $x_v$
for all $v\in V$.
Since $2|E|/|V|$ is the average degree of $G$,
storage requirement is $\mathcal{O}(|V|+|E|)$ when simplified.

Let each node have a set $N\subseteq\{1,\dots,C\}$
such that $c\in N$ if there is a neighbor connected through an edge with color $c$.
Write $N^\mathsf{c}=\{1,\dots,C\}\backslash N$.
With this notation, we can rewrite the algorithm in a simpler, vertex-centric manner:
\begin{align*}
&\text{FOR EACH node}\\
&\quad x^{1/2}= s_{\alpha\lambda_1}\left(\frac{1}{C}
\left(|N^\mathsf{c}|z_{c'}+
\sum_{c\in N}z_{c}^k\right)\right)\\
&\quad\text{FOR EACH color $c\in S$}\\
&\quad\quad z_{c}^{1/2}=2x^{1/2}-z_{c}-\alpha \nabla f(x^{1/2})\\
&\quad\quad\text{Through edge with color $c$,
send $z_{c}^{1/2}$ and receive $z_{c}^{\prime 1/2}$}\\
&\quad\quad s=z_{c}^{1/2}+z_{c}^{\prime 1/2}\\
&\quad\quad d=z_{c}^{1/2}-z_{c}^{\prime 1/2}\\
&\quad\quad x_{c}=(1/2)(s+u_{\alpha\lambda_2}\left(d\right))\\
&\quad\quad z_{c}=z_{c}+x_{c}-x^{1/2},\\
&\quad\text{IF $N^\mathsf{c}\ne \emptyset$}\\
&\quad\quad z_{c'}=x^{1/2}-\alpha \nabla f(x^{1/2})
\end{align*}

\vspace{0.1in}\noindent\textbf{SVM.}
We solve the
standard (primal) support vector machine setup \cite{cortes1995}
\begin{equation}
\mbox{minimize}
\quad
\frac{\lambda}{2}
\|x\|_2^2+
\frac{1}{n}
\sum^n_{i=1}g_i(x),
\label{eq:svm}
\end{equation}
where $x\in \reals^d$ is the optimization variable.
The problem data is embedded in
\[g_i(x)=\max\{1-y_ia_i^Tx,0\}
\]
where $\lambda>0$ is a regularization parameter and
$a_i\in \reals^d$, $b_i\in \reals$ and $y_i\in\{-1,+1\}$
are problem data for $i=1,\dots,n$.
Applying Lemma~\ref{lm:prox-1d} and working out the details, we get
a closed-form solution for the proximal operator:
 \[
 \prox_{\alpha g_i}(x_0)
 =
x_0+
 \Pi_{[0,\alpha]}\left(
 \frac{1-y_ia_i^Tx_0}
 {\|a_i\|_2^2}
 \right)
 y_i a_i
 \]
 for $i=1,\dots,n$.

We view $r=(\lambda/2)
\|x\|_2^2$ and $f=0$
in the notation of \eqref{eq:main-prob},
and apply \eqref{eq:main-method}:
\begin{align*}
x^{k+1/2}&=
\frac{1}{1+\alpha\lambda}\frac{1}{n}\sum^n_{i=1}z^k_i\\
\beta_i^k&= y_i\Pi_{[0,\alpha]}\left(\frac{1-y_ia_i^T
(2x^{k+1/2} - z_i^k)}{\|a_i\|_2^2}\right)\\
z^{k+1}_i&=x^{k+1/2}+\beta_ia_i,
\end{align*}
where the indices $i$ implicitly run through $i=1,\dots,n$.



\vspace{0.1in}\noindent\textbf{Generalized linear model.}
In the setting of  generalized linear models,
the maximum likelihood estimator
is the solution of the optimization problem
\[
\mbox{minimize}\quad\frac{1}{n}\sum^n_{i=1}\big(A(x_i^T\beta)-T_i x_i^T\beta\big),
\]
where $\beta\in \reals^d$ is the optimization variable,
$x_i\in\reals^d$ and $T_i\in \reals$ for $i=1,\dots,n$ are problem data,
and $A$ is a convex function on $\reals$.

We view $r=0$, $f=0$,
and $g_i(\beta)=A(x_i^T\beta)-T_i x_i^T\beta$
in the notation of \eqref{eq:main-prob},
and apply \eqref{eq:main-method}:
\begin{align*}
\beta^{k+1}&=\frac{1}{n}\sum^n_{i=1}z^k_i\\
z^{k+1}_i
&=z^k_i
+\prox_{\alpha g_i}(2\beta^k-z_i^k)-\beta^k
\end{align*}
where the indices $i$ implicitly run from  $i=1,\dots, n$.

For an introduction on generalized linear models, see \cite{mccullagh1989}.

\vspace{0.1in}\noindent\textbf{Network Utility Maximization.}
In the problem of network utility maximization, one solves the optimization problem
\[
\begin{array}{ll}
\mbox{minimize}&
    (1/n)\sum^n_{i=1} f_i(x_i)\\
\mbox{subject to}
&x_i\in X_i\quad i=1,\dots,n\\
& A_ix_i\le y\quad i=1,\dots,n\\
&y\in Y,
\end{array}
\]
where
$f_1,\dots,f_n$ are functions,
$A_1,\dots,A_n$ are matrices,
$X_1,\dots,X_n,Y$ are sets,
and $x_1,\dots,x_n,y$ are the optimization variables
(For convenience, we convert the maximization problem into a minimization problem.)
For a comprehensive discussion on network utility maximization, see \cite{palomar2007}.

This optimization problem is equivalent to the master problem
\[
\mbox{minimize}\quad
    \frac{1}{n}\sum^n_{i=1} g_i(y)+I_Y(y)
\]
where we define the subproblems
\[
g_i(y)=\inf
\left\{
f_i(x_i)\,|\,x_i\in X_i,\, A_ix_i\le y\right\},
\]
for $i=1,\dots,n$. The master problem only involves the variable $y$, and $x_i$ is the variable in the $i$th subproblem.
For each $i=1,\dots,n$,
the function $g_i$ may not be differentiable, but it is convex if $f_i$ and $X_i$ are convex.
If $Y$ is convex, then the equivalent problem is convex.

We view $f=0$ and $r=I_Y$
in the notation of \eqref{eq:main-prob},
and apply \eqref{eq:main-method}:
\begin{align*}
x^{k+1/2}&=\Pi_Y\left(\frac{1}{n}\sum^n_{i=1}z_i^k\right)\\
y_i^{k+1}&
=\argmin_{\substack{x_i\in X_i\\ A_ix_i\le y\\y\in Y}}
\left\{
\alpha f_i(x_i)+\frac{1}{2}\|y-(2x^{k+1/2}-z_i^k)\|_2^2\,
\right\}\\
z_i^{k+1}&=z_i^k+y_i^{k+1}-y^{k+1/2}.
\end{align*}
Network utility maximization is often performed on a
distributed computing network,
and if so the optimization problem for evaluating
the proximal operators can be solved in a distributed, parallel fashion.

\begin{figure}
\begin{center}
\includegraphics[width=.9\textwidth]{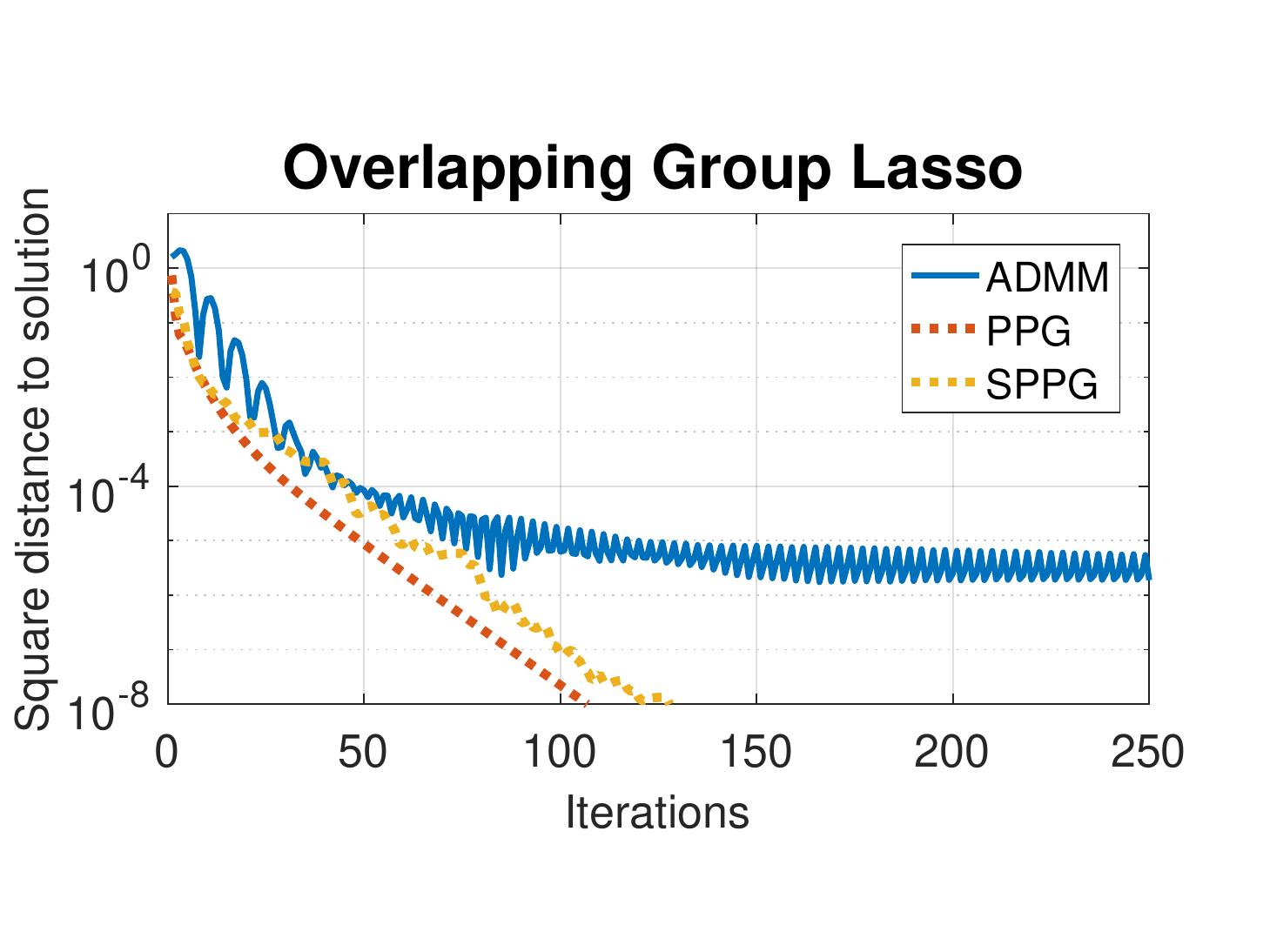}
\caption{
The error $\|x^{k+1/2}-x^\star\|^2$ vs.\ iteration
for overlapping group lasso example.
}
\label{fig:g_lasso}
\end{center}
\end{figure}

\section{Experiments}
\label{sc:experiments}
In this section, we present numerical experiments on two applications
discussed in Section~\ref{sc:applications}.
The first experiment is small and is meant to serve as a proof of concept.
The second experiment is more serious; the problem size is large and we compare the performance
with existing methods.
For the sake of scientific reproducibility, we provide the code used to generate these experiments.

For both experiments, we observe linear convergence.
This is a pleasant surprise, as the theory presented in Section~\ref{sc:conv}
and proved in Section~\ref{s:theory}
only guarantee a $\mathcal{O}(1/k)$ rate.



\vspace{0.1in}\noindent\textbf{Overlapping group lasso.}
The problem size of this setup is $m=300$, $d=42$, $n=3$.
The groups are as described in Section~\ref{sc:applications}.

The dominant cost per iteration of
\eqref{eq:main-method}
is evaluating $\prox_{\alpha r}$
which takes $\mathcal{O}(d^2)$ time
with the precomputed factorization.
Since the cost per iteration of
\eqref{eq:main-sppg}
is no cheaper than
that of \eqref{eq:main-method},
there is no reason to use \eqref{eq:main-sppg}.


We compare the performance of
\eqref{eq:main-method}
to consensus ADMM (cf.\ \S7.1 of \cite{boyd2011}).
Both methods use the same computational subroutines
and therefore have essentially the
same computational cost per iteration.
We show the results in Figure~\ref{fig:g_lasso}.

\vspace{0.1in}\noindent\textbf{SVM.}
The problem size of this setup is $n= 2^{17}=131,072$ and
$d = 512$.
The synthetic dataset $A$ and $y$ are randomly genearated
and the regularization parameter $\lambda=0.1$ is used.
So the problem data $A$ consists of $64\times 2^{20}$ numbers
and requires
$500$MB storage to store in double-precision floating-point format.

First, we compare the performance of
\eqref{eq:main-method} and \eqref{eq:main-sppg}
to the stochastic proximal iteration with diminishing step size $\alpha_k=C/k$
in Figure~\ref{fig:svm1}.
For all three methods, the parameters were roughly tuned for optimal performance.

\begin{figure}
\begin{center}
\includegraphics[width=.9\textwidth]{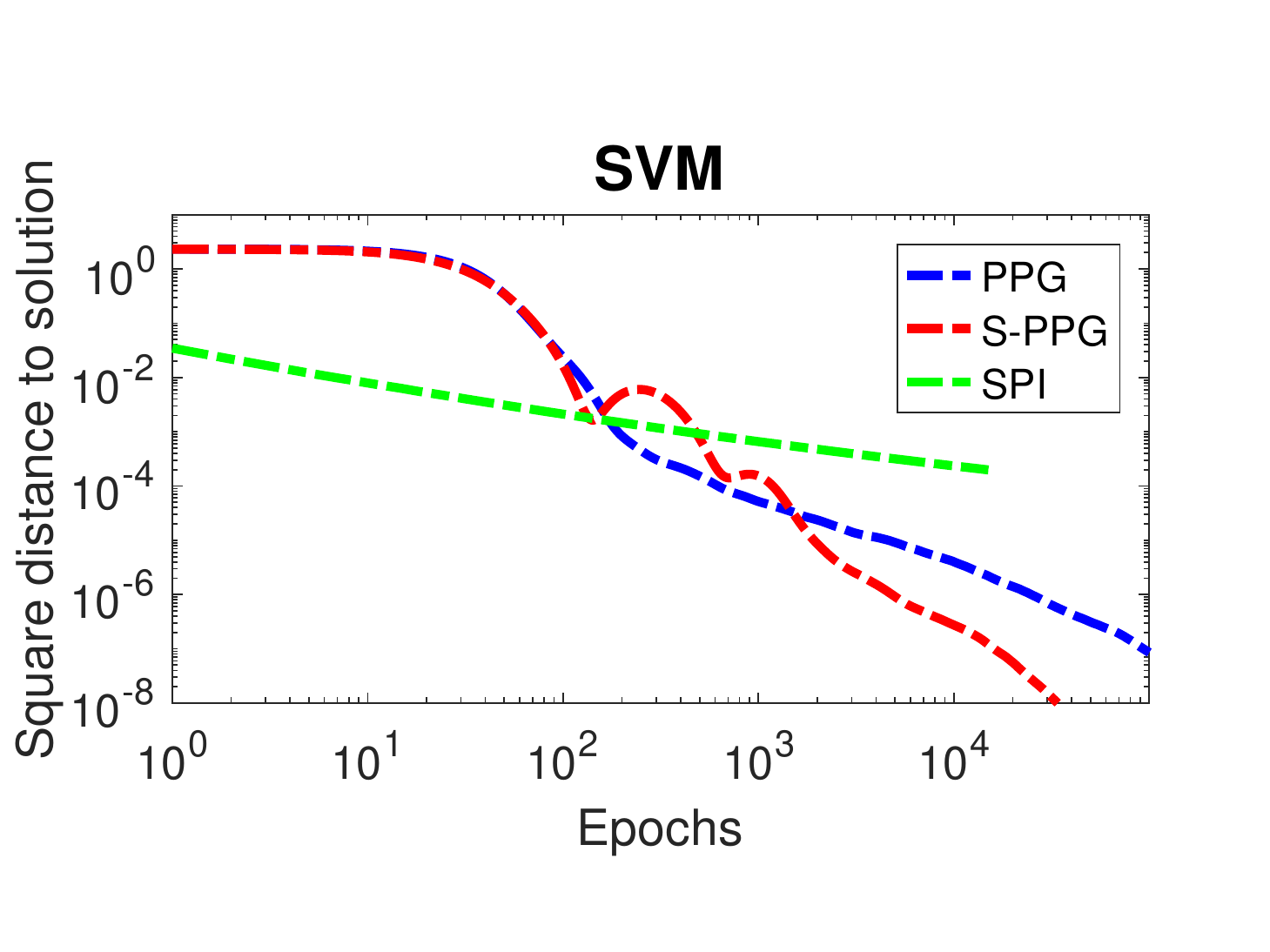}
\caption{
The error $\|x^{k+1/2}-x^\star\|^2$ vs.\ iteration
for the SVM example.
}
\label{fig:svm1}
\end{center}
\end{figure}

Next, we compare the \eqref{eq:main-method} and \eqref{eq:main-sppg}
to a state-of-the-art SVM solver LIBLINEAR \cite{fan2008}.
Since LIBLINEAR is based on a second order method
while \eqref{eq:main-method} and \eqref{eq:main-sppg} are first-order methods,
comparing the number of iterations is not very meaningful.
Rather we compare the wall-clock time these methods take to reach an equivalent level of accuracy.
To compare the quality of the solutions, we use the objective value 
of the problem \eqref{eq:svm}.

The CPU code for \eqref{eq:main-method} and \eqref{eq:main-sppg} are written in C++.
The code is serial, not heaviliy optimized, and does not utilize BLAS (Basic Linear Algebra Subprograms) libraries or any SIMD instructions.
On the other hand, LIBLINEAR is heavily optimized and does utilize BLAS.

We also implemented \eqref{eq:main-method} on CUDA and ran it on a GPU.
The algorithmic structure of \eqref{eq:main-method} is particularly well-suited for CUDA,
especially when the problem size is large.
Roughly speaking, the $x^{k+1/2}$ update requires a reduce operation,
which can be done effectively on CUDA.
The $z_i^k$ updates are embarassingly parallel,
and can be done very effectively on CUDA.
The threads must globally synchronize twice per iteration:
once before computing the average of $z_i^k$ for $i=1,\dots,n$
and once after $x^{k+1/2}$ has been computed.
Generally speaking, global synchronization on a GPU is expensive, 
but we have empirically verified that 
the computational bottleneck is in the other operations, not the synchronization,
when the problem size is reasonably large.

Table~\ref{table:svm} shows the results.
We see that CPU implementation of \eqref{eq:main-method} and \eqref{eq:main-sppg} 
are competitive with LIBLINEAR,
and could even be faster than LIBLINEAR if the code is further optimized.
On the other hand,
the CUDA implementation of \eqref{eq:main-method}
clearly outperforms LIBLINEAR.
\eqref{eq:main-method} and \eqref{eq:main-sppg} were run until the objective values were good as that of LIBLINEAR.
These expeirments were run on an Intel Core i7-990 GPU and a GeForce GTX TITAN X GPU.

\begin{table}
\begin{center}
    \begin{tabular}{  | c  | c | c | c|  }
    \hline
    Method  &Run time &  Objective value  \\ \hline
       LIBLINEAR &  $8.47s$ & $3.7699$   \\ \hline
       CPU \eqref{eq:main-method} &  $33.9s$ ($30$ iterations) & $3.7364$ \\ \hline
       CPU \eqref{eq:main-sppg}  & $37.2s$ ($30$ epochs) & $3.7364$\\ \hline
       CUDA \eqref{eq:main-method} & $0.68s$ ($30$ iterations) & $3.7364$ \\ \hline
    \end{tabular}
\end{center}
    \caption{Run time as a function of grid size}
    \label{table:svm}
\end{table}

\section{Convergence proofs}
\label{s:theory}

We say a function $f$ is closed convex and proper
if its epigraph
\[
\big\{(x,\alpha)\,\mid\,x\in \reals^d,\,|f(x)|<\infty,\,f(x)\le \alpha\big\}
\]
is a closed subset of $\mathbb{R}^{d+1}$,
$f$ is convex, and $f(x)=-\infty$ nowhere and $f(x)<\infty$ for some $x$.

A closed convex and proper function $f$ has $L$-Lipschitz continuous gradient if
$f$ is differentiable everywhere and
\[
\|\nabla f(x)-\nabla f(y)\|_2\le L\|x-y\|_2
\]
for all $x,y\in \reals^d$.
This holds if and only if a
closed convex and proper function $f$ satisfies
\[
\|\nabla f(x)-\nabla f(y)\|_2^2\le L(\nabla f(x)-\nabla f(y))^T(x-y)
\]
for all $x,y\in \reals^d$, which is known as the Baillon-Haddad Theorem
\cite{baillon1977}.
See \cite{bauschke2010,bauschke2011, ryu2016} for a discussion on this.

Proximal operators are firmly non-expansive. This means
for any closed convex and proper function $f$ and $x,y\in\reals^d$,
\[
\|\prox_f(x)-\prox_f(y)\|_2^2\le
(\prox_f(x)-\prox_f(y))^T(x-y).
\]
By applying Cauchy-Schwartz inequality, we can see that firmly non-expansive
operators are non-expansive.

\vspace{0.1in}\noindent\textbf{Some proximable functions.}
As discussed,
many standard references like \cite{combettes2011, parikh2014}
provide a list proximable functions.
Here we discuss the few we use.

An optimization problem of the form
\[
\begin{array}{ll}
\mbox{minimize}&
f(x)\\
\mbox{subject to}& x\in C,
\end{array}
\]
where $x$ is the optimization variable and $C$ is a constraint set,
can be transformed into the equivalent optimization problem
\[
\begin{array}{ll}
\mbox{minimize}&
f(x)+I_C(x).
\end{array}
\]
The \emph{indicator function} $I_C$ is defined as
\[
I_C(x)=\left\{
\begin{array}{ll}
0&\text{ for }x\in C\\
\infty&\text{ otherwise,}
\end{array}
\right.
\]
and
the proximal operator with respect to $I_C$ is
\[
\prox_{\alpha I_C}(x)=\Pi_C(x)
\]
where $\Pi_C$ is the projection onto $C$ for any $\alpha>0$.
So $I_C$ is proximable if the projection onto $C$
is easy to evaluate.

The following results are well known. The proximal operator with respect to $r(x)=|x|$
is called the scalar soft-thresholding operator
\[
\prox_{\lambda r}(x)=
s_\lambda(x) =
\left\{
\begin{array}{ll}
x+\lambda& x<-\lambda\\
0&-\lambda\le x\le \lambda\\
x-\lambda &x> \lambda.
\end{array}
\right.
\]

The proximal operator with respect to $r(x)=\|x\|_2$
is called the vector soft-thresholding operator
\[
\prox_{\lambda r}(x)
=u_\lambda(x) =
\left\{
\begin{array}{ll}
\max\{1-\lambda/\|x\|_2,0\}x&\text{for }x\ne 0\\
0&\text{otherwise}.
\end{array}
\right.
\]

The proximal operator with respect to
$r(M)=\|M\|_*$ is called the matrix
soft-thresholding operator
\[
\prox_{\lambda r}(M)
=
t_\lambda (M)
=\{Us_\lambda (\Sigma) V^T\,|\,U\Sigma V^T=M\text{ is the SVD}
\},
\]
where $s_\lambda (\Sigma)$ is applied element-wise to the diagonals.

\begin{lemma}
\label{lm:prox-1d}
Assume $g(x)=f(a^Tx)$ where $a\in \reals^d$ and $f$ is a closed, convex, and proper function on $\reals$.
Then the $\prox_g$ can be evaluated by solving a one-dimensional optimization problem.
\end{lemma}
\begin{proof}
By examining the optimization problem that defines $\prox_g$
\[
\prox_g(x_0)=
\argmin_x\left\{
f(a^Tx)+\frac{1}{2}\|x-x_0\|_2^2
\right\}
\]
we see that solution must be of the form
$x_0+\beta a$. So
\[
\prox_g(x_0)=x_0+\beta a,
\qquad
\beta=
\argmin_\beta\left\{f(a^Tx_0+\beta \|a\|_2^2)+\frac{\|a\|_2^2}{2}\beta^2\right\}.
\]
\end{proof}
\begin{lemma}
\label{lm:prox-lemma}
Let
\[
g(x_1,x_2,\dots,x_n)=f(a_1x_1+a_2x_2+\dots+a_nx_n)
\]
where $x_1,\dots,x_n\in \reals^d$, $a\in \reals^n$, $a\ne 0$,
and $f:\reals\rightarrow \reals\cup\{\infty\}$ is closed, convex, and proper.
Then we can compute $\prox_g$ with
\begin{gather*}
w=\prox_{\|a\|_2^2f}(a_1\xi_1+a_2\xi_2+\dots +a_n\xi_n)\\
v = \frac{1}{\|a\|_2^2}(a_1\xi_1+a_2\xi_2+\dots+a_n\xi_n-w)\\
\prox_g(\xi_1,\xi_2,\dots,\xi_n)=
\begin{pmatrix}
\xi_1-a_1v\\
\xi_2-a_2v\\
\vdots\\
\xi_n-a_nv
\end{pmatrix}.
\end{gather*}
\end{lemma}
\begin{proof}
The optimality conditions of
$\prox_g(\xi_1,\xi_2,\dots,\xi_n)$
gives us
\begin{align*}
0&\in a_1^2v+a_1(x_1-\xi_1)\\\
&\quad\vdots\qquad\qquad\qquad\qquad\vdots\\
0&\in a_n^2v+a_n(x_n-\xi_n)
\end{align*}
for some $v\in \partial f(a_1x_1+\dots +a_nx_n)$.
Summing this we get
\begin{align*}
0&\in \|a\|_2^2v+
(a_1x_1+\dots +a_nx_n)
-
(a_1\xi_1+\dots +a_n\xi_n)
\end{align*}
and with $w=a_1x_1+\dots +a_nx_n$ we have
\[
w=\prox_{\|a\|_2^2f}(a_1\xi_1+a_2\xi_2+\dots +a_n\xi_n).
\]
The expression for $v$ and $\prox_g(\xi_1,\xi_2,\dots,\xi_n)$ follows from
reorganizing the equations.

\end{proof}
So if $g(x,y)=f(x+y)$ then
\[
\prox_g(x_0,y_0)=
\frac{1}{2}
\begin{pmatrix}
x_0-y_0+\prox_{2f}(x_0+y_0)\\
y_0-x_0+\prox_{2f}(x_0+y_0)
\end{pmatrix}.
\]
If $g(x,y)=f(x-y)$ then
\[
\prox_g(x_0,y_0)=
\frac{1}{2}
\begin{pmatrix}
x_0+y_0+
\prox_{2f}(x_0-y_0)\\
x_0+y_0-\prox_{2f}(x_0-y_0)
\end{pmatrix}.
\]


\subsection{Deterministic analysis}
Let $h$ be a closed, convex, and proper function on $\reals^d$.
When
\[
x=\prox_{\alpha h}(x_0),
\]
we have
\[
\alpha u+x=x_0
\]
with $u\in \partial h(x_0)$.
To simplify the notation, we write
\[
\alpha \tnabla h(x_0)+x=x_0
\]
where $\tnabla h(x_0)\in \partial h(x_0)$.
So is $\tnabla h(x_0)$ a subgradient of $h$ at $x_0$,
and which subgradient $\tnabla h(x_0)$ is referring to
depends on the context.
(This notation is convenient yet potentially sloppy, but we promise to not commit
any fallacy of equivocation.)

\begin{proof}[Proof of Lemma~\ref{lm:encoding}]
Assume $\vz^\star$ is a fixed point of \eqref{eq:main-method} or \eqref{eq:main-sppg}.
Then $p(\vz^\star)=0$.
Then we have
$x^\star=x^{\prime\star}_i$ for $i=1,\dots,n$
where $x^{\star}$ and $x^{\prime\star}_1,\dots,x^{\prime\star}_n$ are as defined in \eqref{eq:p-mapping}.
So
\begin{align*}
0&=\alpha  \tnabla r(x^{\star})+x^{\star}-\bar{z}\\
0&=\alpha  \tnabla g_1(x^{\star})-x^{\star}+z_1^{\star}+\alpha \nabla f_1(x^{\star})\\
&\qquad\qquad\vdots\qquad\qquad\qquad \vdots\\
0&=\alpha  \tnabla g_n(x^{\star})-x^{\star}+z_n^{\star}+\alpha \nabla f_n(x^{\star}).
\end{align*}
Adding these up and dividing by $n$ appropriately gives us
\[
0=\tnabla r(x^{\star})+\tnabla \bar{g}(x^{\star})+\nabla \bar{f}(x^{\star}),
\]
so $x^{\star}$ is a solution of Problem~\eqref{eq:main-prob}.
Reorganize the definition of $x$ in \eqref{eq:p-mapping}
to get
\[
x^{\star}=\argmin_x
\left\{
r(x)-\frac{1}{\alpha}(\bar{z}^{\star}-x^{\star})^Tx+\frac{1}{2\alpha}\|x-x^{\star}\|_2^2
\right\}.
\]
So $x^{\star}$ minimizes $L(\cdot,\vx^{\star},(1/\alpha)(\vz^{\star}-\vx^{\star}))$,
where $L$ is defined in \eqref{eq:lagrangian}.
Reorganize the definition of $x'$ in \eqref{eq:p-mapping}
to get
\[
x^{\star}=\argmin_x\left\{
f(x^{\star})+(\nabla f(x^{\star}))^T(x-x^{\star})
+g(x)+
\frac{1}{\alpha}(z_i^{\star}-x^{\star})^Tx+
\frac{1}{2\alpha}\|x-x^{\star}\|_2^2
\right\}.
\]
So $x^{\star}$ minimizes $L(x^{\star},\cdot,(1/\alpha)(\vz^{\star}-\vx^{\star}))$.
So
$\boldsymbol{\nu}^{\star}=(1/\alpha)(\vz^{\star}-\vx^{\star})$
is a dual solution of Problem~\eqref{eq:main-prob}.

The argument works in the other direction as well.
If we assume $(x^\star,\boldsymbol{\nu}^{\star})$
is a primal dual solution of Problem~\eqref{eq:main-prob},
we can show that $\vz^{\star}=\vx^{\star}+\alpha \boldsymbol{\nu}^{\star}$
is a fixed point of \eqref{eq:main-method}
by following a similar line of logic.
\end{proof}

Before we proceed to the main proofs, we introduce more notation.
Define the function
\[
\brr(\vx) = \frac{1}{n}\sum_{i=1}^n r(x_i) + I_{C}(\vx),
\]
where
\[
C=\{(x_1,\dots,x_n)\,|\,x_1=\dots=x_n\}.
\]
So
\[
I_{C}(\vx)=
\left\{
\begin{array}{ll}
0 & \text{for }x_1=\cdots=x_n\\
\infty&\text{otherwise.}
\end{array}
\right.
\]
As before, we write
\[
\brf(\vx)=
\frac{1}{n}\sum_{i=1}^n f_i(x_i),
\qquad \brg(\vx) = \frac{1}{n}\sum_{i=1}^n g_i(x_i).
\]
With this new notation, we can recast
Problem~\eqref{eq:main-prob} into
\[
\begin{array}{ll}
\mbox{minimize}&
\brr(\vx)+\brf(\vx) +\brg(\vx),
\end{array}
\]
where $\vx\in \reals^{dn}$ is the optimization variable.
We can also rewrite the definition of $p$ as the following three-step process, which starts from $\vz$, produces intermediate points $\vx,\vx'$, and yields $p(\vz)$:
\begin{subequations}\label{eq:pmap2}
\begin{align}
\vx & = \prox_{\alpha \brr}(\vz)\label{eq:pmap2x}\\
\vx' & = \prox_{\alpha \brg}\left(2\vx-\vz-\alpha \nabla \brf(\vx)\right)\label{eq:pmap2xp}\\
p(\vz)& = (1/\alpha)(\vx-\vx').
\label{eq:p-mapping2}
\end{align}
\end{subequations}
Note that $x_1,\dots,x_n$ in $\vx$ out of \eqref{eq:pmap2x} are identical due to $I_C$.
This is the same $p$ as the $p$ defined in \eqref{eq:p-mapping};
we're just using the new notation.

We treat the boldface variables as vectors in $\reals^{dn}$.
So the inner product between boldface variables is
\[
\vx^T\vz = \sum^n_{i=1}x_i^Tz_i
\]
and the gradient of $\bar{f}(\vx)$ is
\[
\nabla \bar{f}(\vx)=
\frac{1}{n}
\begin{bmatrix}
\nabla f_1(x_1)\\
\nabla f_2(x_2)\\
\vdots\\
\nabla f_n(x_n)
\end{bmatrix}.
\]
We use $\tnabla g(\vx)$ and $\tnabla r(\vx)$ in the same manner as before.

\begin{lemma}
\label{lem:key}
Let $\vz$ and $\tilde{\vz}$ be any points in $\reals^{dn}$. Then
\[
\alpha\|p(\vz)-p(\tilde{\vz})\|^2 \le  ( p(\vz)-p(\tilde{\vz}))^T(\vz-\tilde{\vz})
-
( \nabla \bar{f}(\vx)-\bar{f}(\tilde{\vx}))^T( \vx'-\tilde{\vx}')
\]
where $\tilde{\vx}$ and $\tilde{\vx}'$ are obtained by applying \eqref{eq:pmap2x} and then \eqref{eq:pmap2xp}, respectively, to $\tilde{\vz}$ instead of $\vz$.
\end{lemma}

\begin{proof}
This Lemma is similar to
Lemma 3.3 of \cite{davis2017}.
We reproduce the proof using this paper's notation.
\begin{align*}
\|\vz-\alpha p(\vz)-\tilde{\vz}+\alpha p(\tilde{\vz})\|_2^2
&\overset{(a)}{=}
\|\vz-\vx-\tilde{\vz}+\tilde{\vx}\|_2^2+
\|\vx'-\tilde{\vx}'\|_2^2
+2(\vz-\vx-\tilde{\vz}+\tilde{\vx})^T(\vx'-\tilde{\vx}')\\
&\overset{(b)}{\le}
(\vz-\vx-\tilde{\vz}+\tilde{\vx})^T
(\vz-\tilde{\vz})
+
(\vx'-\tilde{\vx}')^T
(2\vx-\vz-\alpha \nabla f(\vx)-2\tilde{\vx}+\tilde{\vz}+\alpha \nabla f(\tilde{\vx}))\\
&\qquad +2(\vz-\vx-\tilde{\vz}+\tilde{\vx})^T(\vx'-\tilde{\vx}')\\
&=
(\vz-\alpha p(\vz)-\tilde{\vz}+\alpha p(\tilde{\vz}))^T(\vz-\tilde{\vz})
-\alpha(\nabla f(\vx)-\nabla f(\tilde{\vx}))^T(\vx'-\tilde{\vx}'),\\
\end{align*}
where (a) is due to $\alpha p(\vz) = \vx-\vx'$ and $\alpha p(\tilde{\vz}) = \tilde{\vx}-\tilde{\vx'}$,
(b) follows from the fact that
the two mappings:
\[
\vz~ \mapsto~ \vz-\vx=\vz-\prox_{\alpha \brr}(\vz)
\]
and
\[
2\vx-\vz-\alpha \nabla \brf(\vx) ~\mapsto~ \vx'
=
\prox_{\alpha \brg}\left(2\vx-\vz-\alpha \nabla \brf(\vx)\right)
\]
are both firmly non-expansive.
By expanding the $\|\vz-\alpha p(\vz)-\tilde{\vz}+\alpha p(\tilde{\vz})\|_2^2$ and cancellation, we obtain
\begin{align*}
\|\alpha p(\vz)-\alpha p(\tilde{\vz})\|_2^2
&\le
(\alpha p(\vz)- \alpha p(\tilde{\vz}))^T(\vz-\tilde{\vz})
-\alpha(\nabla f(\vx)-\nabla f(\tilde{\vx}))^T(\vx'-\tilde{\vx}'),
\end{align*}
which proves the lemma by dividing both sides by $\alpha$.
\end{proof}

\begin{lemma}
\label{lem:descent}
Let $\vz^\star$ be any fixed point of
\eqref{eq:main-method}, i.e., $p(\vz^\star)=0$.
Then
\begin{equation}
\|\vz^k-\vz^\star\|_2^2\le \|\vz^0-\vz^\star\|_2^2
\label{eq:bounded-z}
\end{equation}
for all $k=0,1,\dots$.
Moreover, we have
\begin{align}
  &\sum_{k=0}^{\infty}\|p(\vz^k)\|^2<\infty
  \label{eq:p-sum}  \\
  &\sum_{k=0}^{\infty}\| \nabla \brf(\vx^{k+1/2})-\nabla \brf(\vx^\star)\|^2<\infty.
\end{align}
Finally,
$\|p(\vz^{k})\|^2$ monotonically decreases.
\end{lemma}
\begin{proof}
With the Baillon-Haddad theorem and Young's inequality on (a) we get
\begin{align}
 -(\nabla &\brf(\vx)-\nabla \brf(\tilde{\vx}))^T (\vx'-\tilde{\vx}')\\
&= -( \nabla \brf(\vx)-\nabla \brf(\tilde{\vx}))^T
(\vx-\alpha p(\vz)-\tilde{\vx}+\alpha p(\tilde{\vz}))\\
&=-(\nabla \brf(\vx)-\nabla \brf(\tilde{\vx}))^T
(\vx-\tilde{\vx}) + \alpha(   \nabla \brf(\vx)-\nabla \brf(\tilde{\vx}))^T
( p(\vz)- p(\tilde{\vz}))\\
&\overset{(a)}{\le} -\tfrac{1}{L}\| \nabla \brf(\vx)-\nabla \brf(\tilde{\vx})\|^2 + \tfrac{3}{4L}\|\nabla \brf(\vx)-\nabla \brf(\tilde{\vx})\|^2+\tfrac{L\alpha^2}{3}\|  p(\vz)- p(\tilde{\vz})\|^2\\
& = \tfrac{L\alpha^2}{3}\|  p(\vz)- p(\tilde{\vz})\|^2-\tfrac{1}{4L}\| \nabla \brf(\vx)-\nabla \brf(\tilde{\vx})\|^2.\label{eq:key2}
\end{align}
Combining Lemma~\ref{lem:key} and equation \eqref{eq:key2}, we obtain
\begin{align}
\label{eq:key3}
-( p(\vz)- p(\tilde{\vz}))^T(\vz-\tilde{\vz}) \le \alpha (\tfrac{\alpha L}{3}-1)\| p(\vz)- p(\tilde{\vz})\|^2-\tfrac{1}{4L}\| \nabla \brf(\vx)-\nabla \brf(\tilde{\vx})\|^2.
\end{align}

Applying \eqref{eq:key3} separately with $(\vz,\tilde{\vz})=(\vz^k,\vz^\star)$ and $(\vz^k,\vz^{k+1})$,
we get, respectively,
\begin{align}
\|\vz^{k+1}-\vz^\star\|^2 
&= \|\vz^k-\vz^\star\|^2+\alpha^2\|p(\vz^k)\|^2 - 2\alpha (\vz^k-\vz^\star)^Tp(\vz^k)
\nonumber
\\
& \le \|\vz^k-\vz^\star\|^2-\alpha^2(1-\tfrac{2\alpha L}{3})\|p(\vz^k)\|^2-\tfrac{\alpha}{2L}\| \nabla \brf(\vx^{k+1/2})-\nabla \brf(\vx^\star)\|^2
\label{eq:keyz}
\end{align}
and
\begin{align}
\|p(\vz^{k+1})\|^2
&= \|p(\vz^{k})\|^2 + \|p(\vz^{k+1})-p(\vz^{k})\|^2 - 2( p(\vz^{k+1})-p(\vz^{k}))^T\tfrac{1}{\alpha}(\vz^{k+1}-\vz^k)\nonumber\\
&\le \|p(\vz^{k})\|^2 -(1-\tfrac{2\alpha L}{3}) \|p(\vz^{k+1})-p(\vz^{k})\|^2.\label{eq:keyp}
\end{align}
(In \eqref{eq:key2}
we can use a different parameter for Young's inequality,
and improve  inequalities \eqref{eq:keyz} and \eqref{eq:keyp}
to allow $\alpha<2/L$, which is better than $\alpha< 3/(2L)$.
In fact, $\alpha<2/L$ is sufficient for convergence in Theorem~\ref{thm:seq}. However,
we use the current version because
we need the last term of inequality \eqref{eq:keyz} for proving
Theorem~\ref{thm:erg}.)

Summing \eqref{eq:keyz} through $k=0,1,\dots$ give us the summability result.
Inequality \eqref{eq:keyp} states that
$\|p(\vz^{k})\|^2$ monotonically decreases.

\end{proof}

\begin{proof}[Proof of Theorem~\ref{thm:seq}]
Lemma~\ref{lem:descent} already states that $\|p(\vz^k)\|_2^2$
decreases monotonically.
Using inequality \eqref{eq:p-sum} of Lemma~\ref{lem:descent},
we get
\[
\|p(\vz^k)\|^2_2=
\min_{i=0,1,\dots,k}\|p(\vz^i)\|^2_2\le
\frac{1}{k}
\sum_{i=0}^{\infty}\|p(\vz^i)\|^2_2=C/k=\mathcal{O}(1/k)
\]
for some finite constant $C$.
(The rate $\mathcal{O}(1/k)$ can be improved to $o(1/k)$
using, say, Lemma 1.2 of \cite{deng2016},
but we present the simpler argument.)

By \eqref{eq:bounded-z} of Lemma~\ref{lem:descent},
$\vz^k$ is a bounded sequence
and will have a limit point, which we call $\vz^\infty$.
Lemma~\ref{lem:key} also implies $p$ is a continuous function.
Since $p(\vz^k)\rightarrow 0$ and $p$ is continuous,
the limit point $\vz^\infty$ must satisfy
$p(\vz^\infty)=0$.
Applying inequality \eqref{eq:bounded-z} with $\vz^\star=\vz^\infty$
tells us that
$\|\vz^k-\vz^\infty\|_2^2\rightarrow 0$,
i.e., the entire sequence converges.

Since
$\prox_{\alpha r}$ is a continuous function
\[
x^{k+1/2}=\prox_{\alpha r}(\bar{z}^k)
\rightarrow
\prox_{\alpha r}(\bar{z}^\star)=x^\star.
\]
With this same argument, we also conclude that
$x^k_i\rightarrow x^\star$ for all $i=1,\dots,n$.
\end{proof}

\begin{proof}[Proof of Theorem~\ref{thm:fun-rate}]
A convex function $h$ satisfies the inequality
\[
h(x)-h(\tilde{x})\le (\tnabla h(x))^T(x-\tilde{x})
\]
for any $x$ and $\tilde{x}$ (so long as a subgradient $\tnabla h(x)$
exists).

Applying this inequality we get
\begin{align}
&E^k\le (\tnabla\brr(\vx^{k+1/2})+\nabla\brf(\vx^{k+1/2}))^T( \vx^{k+1/2}-\vx^\star ) +
( \tnabla\brg(\vx^{k+1}))^T(\vx^{k+1}-\vx^\star  )\nonumber\\
&= (\tnabla\brr(\vx^{k+1/2})+\nabla\brf(\vx^{k+1/2})+\tnabla\brg(\vx^{k+1}))^T(\vx^{k+1/2}-\vx^\star )
- ( \tnabla\brg(\vx^{k+1}))^T(\alpha p(\vz^k) )\nonumber\\
&=( \vx^{k+1/2}-\alpha \tnabla\brg(\vx^{k+1})-\vx^\star)^Tp(\vz^k) \\
&=( (\vz^{k+1}-\vz^\star)+ \alpha (\tnabla\brr(\vx^\star)+\nabla\brf(\vx^{k+1/2})))^Tp(\vz^k)\\
&=( \vz^{k+1}-\vz^\star)^Tp(\vz^k)+
\alpha ( \tnabla\brr(\vx^\star)+\nabla\brf(\vx^{k+1/2}))^Tp(\vz^k).\label{eq:keyEk1}
\end{align}
For the second equality, we used
\[
p(\vz^k)=\tnabla\brr(\vx^{k+1/2})+\nabla\brf(\vx^{k+1/2})+\tnabla\brg(\vx^{k+1}).
\]
and combined terms. 
For the third equality, we used
$\vx^\star = \vz^\star-\alpha \tnabla\brr(\vx^\star)$
and
\begin{align*}
\vx^{k+1/2}-&\alpha\nabla \brf(\vx^{k+1/2})-\alpha \tnabla\brg(\vx^{k+1})\\
&= \vz^k - \alpha \tnabla\brr(\vx^{k+1/2})-\alpha\nabla \brf(\vx^{k+1/2}) - \alpha \tnabla\brg(\vx^{k+1})\\
&= \vz^k - \alpha p(\vz^k) \\
&= \vz^{k+1}.
\end{align*}
Likewise we have
\begin{align}
&E^k \ge ( \tnabla\brr(\vx^\star)+\nabla\brf(\vx^\star))^T(\vx^{k+1/2}-\vx^\star )
+ ( \tnabla\brg(\vx^\star))^T(\vx^{k+1}-\vx^\star )\nonumber\\
&=(\vx^{k+1}-\vx^\star)^Tp(\vz^\star) + ( \tnabla \brr(\vx^\star)+\nabla\brf(\vx^\star))^T\alpha p(\vz^k)\\
& =(\tnabla \brr(\vx^\star)+\nabla\brf(\vx^\star))^T\alpha p(\vz^k).
\label{eq:keyEk2}
\end{align}
Here we use
\[
p(\vz^\star)=\tnabla\brr(\vx^\star)+\nabla\brf(\vx^\star)
+\tnabla \brr(\vx^\star)=0
\]

Theorem~\ref{thm:seq}
states that
the sequences
$\vz^1,\vz^2,\dots$ and $\vx^{1+1/2},\vx^{2+1/2},\dots$ both converge and therefore are bounded
and that $\|p(\vz^k)\|_2=\mathcal{O}(1/\sqrt{k})$.
Combining this with the bounds on $E^k$
gives us $|E^k|\le \mathcal{O}(1/\sqrt{k})$.

\end{proof}


\begin{proof}[Proof of Theorem~\ref{thm:erg}]
By Jensen's inequality, we have
\begin{align}\label{eq:jensens}
&E^k_\mathrm{erg} \le\frac{1}{k}\sum_{i=1}^k E^i.
\end{align}
Continuing the last line of \eqref{eq:keyEk1}, we get
\begin{align}\label{eq:eachEk}
  E^k&\le
  ( \vz^{k+1}-\vz^\star)^Tp(\vz^k)+
\alpha ( \tnabla\brr(\vx^\star)+\nabla\brf(\vx^{k+1/2}))^Tp(\vz^k) \\
&=\tfrac{1}{2\alpha}\|\vz^k-\vz^\star\|^2-\tfrac{1}{2\alpha}\|\vz^{k+1}-\vz^\star\|^2-\tfrac{\alpha}{2}\|p(\vz^k)\|^2
+
\alpha ( \tnabla\brr(\vx^\star)+\nabla\brf(\vx^{k+1/2}))^Tp(\vz^k).\label{eq:telscpEk}
\end{align}
Combining \eqref{eq:jensens} and \eqref{eq:eachEk} and after telescopic cancellation,
\begin{align}
E^k_\mathrm{erg} &\le \frac{1}{2\alpha k}\|\vz^1-\vz^\star\|^2 + \frac{1}{k}\sum_{i=1}^k
\alpha ( \tnabla\brr(\vx^\star)+\nabla\brf(\vx^{i+1/2}))^Tp(\vz^i)\\
& = \mathcal{O}(1/k) + \frac{1}{k}
\alpha (\tnabla\brr(\vx^\star)+\nabla\brf(\vx^\star))^T
\sum_{i=1}^kp(\vz^i)
+
\frac{1}{k}
\alpha ( \nabla\brf(\vx^i)-\nabla\brf(\vx^\star))^T
\sum_{i=1}^kp(\vz^i)\\
& \le  \mathcal{O}(1/k) + \frac{1}{k\alpha} \|\vz^{k+1}-\vz^1\|\|\tnabla\brr(\vx^\star)+\nabla\brf(\vx^\star)\|
+ \frac{1}{2k}\sum_{i=1}^k \big(\alpha^2\|p(\vz^i)\|^2 + \|\nabla\brf(\vx^i)-\nabla\brf(\vx^\star)\|^2\big)\\
& =\mathcal{O}(1/k),
\end{align}
where the last line holds due to the boundedness of $\vz^k$
and Lemma~\ref{lem:descent}.
With a similar argument as in \eqref{eq:keyEk2},
we get
\begin{align}
  E^k_\mathrm{erg} &\ge
  (\tnabla \brr(\vx^\star)+\nabla\brf(\vx^\star))^T(
  \vx^{k+1/2}_\mathrm{erg}-\vx^{k+1}_\mathrm{erg})
  =
\frac{1}{k} \left(\sum_{i=1}^{k}\alpha p(\vz^k)
\right)^T
(\tnabla \brr(\vx^\star)+\nabla \brf(\vx^\star))\\
&= \tfrac{1}{k}(\vz^{k}-\vz^0)^T( \tnabla \brr(\vx^\star)+\nabla \brf(\vx^\star))
= O(1/k),
  \label{eq:nbEkbnd}
\end{align}
where have once again used the boundedness of $(\vz^k)_k$.
Furthermore, since
\begin{align}
|E^k_\mathrm{erg}-e^k_{\mathrm{erg}}|\le
L_g\|
\vx^{k+1}_\mathrm{erg}
-\vx^{k+1/2}_\mathrm{erg}\| =
L_g
\|\tfrac{1}{k}(\vz^{k+1}-\vz^1)\|=
\mathcal{O}(1/k), \label{eq:bEkbek}
\end{align}
we have $e^k_{\mathrm{erg}} =\mathcal{O}(1/k)$.

\end{proof}

\subsection{Stochastic analysis}

\begin{proof}[Proof of Theorem~\ref{thm:stc-conv}]
To express the updates of \eqref{eq:main-sppg}
we introduce the following notation:
\[
p(\vz^k)_{[i]} =
\begin{bmatrix}
0\\
\vdots\\
p(\vz^k)_{i}\\
\vdots\\0
\end{bmatrix}.
\]
With this notation, we can express
the iterates of \eqref{eq:main-sppg} as
\begin{align}\label{eq:stomain}
\vz^{k+1}=\vz^k-\alpha p(\vz^k)_{[i(k)]},
\end{align}
and we also have conditional expectations
\begin{align*}
  \EE_k p(\vz^k)_{[i(k)]} &= \frac{1}{n}p(\vz^k),\\
  \EE_k \|p(\vz^k)_{[i(k)]}\|^2 &= \frac{1}{n}\|p(\vz^k)\|^2.
\end{align*}
Here, we let $\EE$ denote the expectation over all random variables $i(1),i(2),\ldots$, and $\EE_k$ denote the expectation over $i(k)$ \emph{conditioned on} $i(1),i(2),\ldots,i(k-1)$.

The convergence of this algorithm has been recently analyzed in \cite{combettes2015_stoch}
when $i(k)$ is chosen at random. Below, we adapt its proof to our setting with new rate results.


Note that
Lemma~\ref{lem:key} and inequality \eqref{eq:key2} remain valid,
as they are not tied to a specific sequence of random samples.
Hence, similar to \eqref{eq:keyz}, we have
\begin{align}
  \|\vz^{k+1}-\vz^*\|^2 
= \|\vz^k-\vz^*\|^2+\alpha^2 \|p(\vz^k)_{[i(k)]}\|^2 - 2\alpha
(\vz^k-\vz^*)^T p(\vz^k)_{[i(k)]}.
\end{align}
We take the conditional expectation to get
\begin{align}
\EE_k\|\vz^{k+1}-\vz^*\|^2 
&= \|\vz^k-\vz^*\|^2+\alpha^2 \EE_k\|p(\vz^k)_{[i(k)]}\|^2 - 2\alpha
(\vz^k-\vz^*)^T\EE_kp(\vz^k)_{[i(k)]}\\[-12pt]
&= \|\vz^k-\vz^*\|^2+\tfrac{\alpha^2}{n}\|p(\vz^k)\|^2 - \tfrac{2\alpha}{n}  (\vz^k-\vz^*)^Tp(\vz^k)\\
&\le \|\vz^k-\vz^*\|^2-\tfrac{\alpha^2}{n}(1-\tfrac{2\alpha L}{3})\|p(\vz^k)\|^2-\tfrac{\alpha}{2Ln}\| \nabla \brf(\vx^k)-\nabla \brf(\vx^*)\|^2.\label{eq:keyEz}
\end{align}
By the same reasoning as before, we have
\[
\min_{0\le i\le k}\EE\|p(\vz^k)\|^2\le \mathcal{O}(1/k).
\]
By \eqref{eq:keyEz}, the sequence $\big(\|\vz^{k}-\vz^*\|^2\big)_{k\ge 0}$ is a nonnegative supermartingale.
Applying Theorem 1 of \cite{robbins1985} to \eqref{eq:keyEz} yields the following three properties, which hold with probability one
for every fixed point $\vz^\star$:
\begin{enumerate}
    \item the squared fixed-point residual sequence is summable, that is,
    \begin{align}\label{eq:summablep}
\sum^\infty_{k=0}\|p(\vz^k)\|^2<\infty,
\end{align}
    \item $\|\vz^{k}-\vz^*\|^2$ converges to a nonnegative random number, and
    \item $(\vz^{k})_{k\ge 0}$ is bounded.
\end{enumerate}

To proceed as before, however, we need
$\|\vz^{k}-\vz^*\|^2$ to converge to a nonnegative random number (not necessarily 0)
for all fixed points $\vz^*$ with probability one.
For each fixed point $\vz^*$,
the previous argument states that there is a measure one event set\footnote{Each event is a randomly realized sequence of iterates $(\vz^k)_{k\ge 0}$ in S-PPG.}, $\Omega(\vz^*)$, such that
$\|\vz^{k}-\vz^*\|^2$ converges for all $(\vz^k)_{k\ge 0}$ taken from $\Omega(\vz^*)$. Note that $\Omega(\vz^*)$ depends on $\vz^*$ because we must select $\vz^*$ to form \eqref{eq:keyEz} first.
Since the number of fixed points (unless there is only one) is uncountable, $\cap_{\text{fixed point}~\vz^*} \Omega(\vz^*)$ may not be measure one.
Indeed, it is measure one as we now argue.

Let $Z^*$ be the set of fixed points.
Since $\mathbb{R}^{dn}$ is \emph{separable} (i.e., containing a countable, dense subset),
$Z^*$ has a countable dense subset, which we write as
$\{\vz^*_1,\vz^*_2,\dots\}$.
By countability, $\Omega_c=\cap_{i=1,2,\dots} \Omega(\vz^*_i)$
is a measure one event set, and it is defined independently of the choice of $\vz^*$. Next we show that $\|\vz^{k}-\vz^*\|^2$ to converge to a nonnegative random number (not necessarily 0)
for all fixed points $\vz^*$ with probability one by $\Omega_c$, that is, $\lim_k\|\vz^{k}-\vz^*\|$ exists for all $\vz^*\in Z^*$ and all $(\vz^k)_{k\ge 0}\in\Omega_c$.

Now consider any $\vz^*\in Z^*$ and $(\vz^k)_{k\ge 0}\in\Omega_c$. Then
for any $\varepsilon>0$, there is a $\vz^*_i$ such that
$\|\vz^*-\vz^*_i\|\le \varepsilon$.
By the triangle inequality, we can bound $\|\vz^{k}-\vz^*\|$ as
\begin{align*}
   \|\vz^{k}-\vz^*\|&\le \|\vz^{k}-\vz_i^*\|+\|\vz_i^*-\vz^*\| \le \|\vz^{k}-\vz_i^*\| + \varepsilon,\\
   \|\vz^{k}-\vz^*\|&\ge \|\vz^{k}-\vz_i^*\|-\|\vz_i^*-\vz^*\| \ge \|\vz^{k}-\vz_i^*\| - \varepsilon.
\end{align*}
Since $\|\vz^{k}-\vz_i^*\|$ converges, we have
we have
\begin{align*}
  \limsup_k\|\vz^{k}-\vz^*\| &\le \varepsilon,\\
  \liminf_k\|\vz^{k}-\vz^*\| &\ge -\varepsilon.
\end{align*}
As $\varepsilon>0$ is arbitrary, $\liminf_k\|\vz^{k}-\vz^*\|=\limsup_k\|\vz^{k}-\vz^*\|$. So, $\lim_k\|\vz^{k}-\vz^*_i\|$ exists.

Finally, we can proceed with the same argument as in the proof of Theorem~\ref{thm:seq},
and conclude that
$\vz^k\to \vz^*$, $\vx^{k+1/2}\rightarrow \vx^*$, and $\vx^{k+1}\rightarrow \vx^*$
on the measure one set $\Omega_c$.

\end{proof}

\begin{proof}[Proof of Theorem \ref{thm:storates}]
\newcommand{\vp}{p}
This proof focuses on the treatments that are different from the deterministic analysis. We only go through the steps of estimating the upper bound of $\EE(E^k)$, skipping the similar treatment to obtain the lower bound and other rates.

We reuse a part of \eqref{eq:keyEk1} but avoid replacing $\vz^{k}-\alpha p(\vz^k)$ by $\vz^{k+1}$ because of \eqref{eq:stomain}:
\begin{align}
E^k&\le
(\vx^{k+1/2}-\alpha \tnabla\brg(\vx^{k+1})-\vx^*)^T p(\vz^k) \\
&=((\vz^{k}-\alpha p(\vz^k)-\vz^*)+ \alpha (\tnabla\brr(\vx^*)+\nabla\brf(\vx^{k+1/2})))^T p(\vz^k)\\
&=( \vz^{k}-\alpha p(\vz^k)-\vz^*)^Tp(\vz^k)+\alpha ( \tnabla\brr(\vx^*)+\nabla\brf(\vx^{k+1/2}))^Tp(\vz^k)\\
&=( \vz^{k}-\alpha p(\vz^k)-\vz^*)^Tp(\vz^k)+\alpha ( \tnabla\brr(\vx^*)+\nabla\brf(\vx^*))^Tp(\vz^k)\\
&\quad + \alpha ( \nabla\brf(\vx^{k+1/2})-\nabla\brf(\vx^*))^Tp(\vz^k).\label{eq:keyEkS}
\end{align}
By the Cauchy-Schwarz inequality,
\begin{align}
  \EE(E^k)&\le \EE\|\vz^{k}-\alpha p(\vz^k)-\vz^*\|\cdot\EE\|p(\vz^k)\| +\alpha\|\tnabla\brr(\vx^*)+\nabla\brf(\vx^*)\|\cdot\EE\|p(\vz^k)\|\\
  &\quad +  \alpha \EE\| \nabla\brf(\vx^{k+1/2})-\nabla\brf(\vx^*)\|\cdot \EE \|p(\vz^k)\| \\
  &\le \bigg(\sqrt{\EE\|\vz^{k}-\alpha p(\vz^k)-\vz^*\|^2}+\alpha\|\tnabla\brr(\vx^*)+\nabla\brf(\vx^*)\|\\
  &\quad +  \alpha \sqrt{\EE\| \nabla\brf(\vx^{k+1/2})-\nabla\brf(\vx^*)\|^2}\bigg)\sqrt{\EE\|p(\vz^k)\|^2}.\label{eq:sqbnd}
\end{align}
Here we have  $\sqrt{\EE\|\vz^{k}-\alpha p(\vz^k)-\vz^*\|^2}\le \sqrt{\|\vz^{0}-\vz^*\|^2}$ since, similar to \eqref{eq:keyz},
\begin{align}
\|&\vz^{k}-\alpha p(\vz^k)-\vz^*\|^2 
= \|\vz^k-\vz^*\|^2+\alpha^2\|p(\vz^k)\|^2 - 2\alpha (\vz^k-\vz^*)^Tp(\vz^k)
\nonumber
\\
&\le \|\vz^k-\vz^*\|^2-\alpha^2(1-\tfrac{2\alpha L}{3})\|p(\vz^k)\|^2-\tfrac{\alpha}{2L}\| \nabla \brf(\vx^{k+1/2})-\nabla \brf(\vx^*)\|^2
\nonumber\\
&\le \|\vz^{k}-\vz^*\|^2
\end{align}
and, by \eqref{eq:keyEz}, $\EE\|\vz^{k}-\vz^*\|^2 \le \EE\|\vz^{k-1}-\vz^*\|^2\le \cdots\le \|\vz^{0}-\vz^*\|^2.$ The next term in \eqref{eq:sqbnd}, $\alpha\|\tnabla\brr(\vx^*)+\nabla\brf(\vx^*)\|$, is a constant. For the third term, from
\begin{align*}
\| \nabla\brf(\vx^{k+1/2})-\nabla\brf(\vx^*)\|^2\overset{(a)}{\le} L^2\| \vx^{k+1/2}-\vx^*\|^2\overset{(b)}\le L^2\| \vz^{k}-\vz^*\|^2,
\end{align*}
where (a) is due to Lipschitz continuity and (b) due to nonexpansiveness of the proximal mapping, it follows that $\alpha \sqrt{\EE\| \nabla\brf(\vx^{k+1/2})-\nabla\brf(\vx^*)\|^2}\le \alpha L  \sqrt{\|\vz^{0}-\vz^*\|^2}$.
Since $\EE\|p(\vz^k)\|^2=\mathcal{O}(1/k)$, we immediately have $\EE(E^k)\le \mathcal{O}(1/\sqrt{k})$. Similarly, we can also show $-\EE(E^k)\le \mathcal{O}(1/\sqrt{k})$. Therefore, $\EE|E^k| = \mathcal{O}(1/\sqrt{k})$.

By extending the previous analysis of $|E^k_\mathrm{erg}|$ and $e^k_\mathrm{erg}$ to $\EE|E^k_\mathrm{erg}|$ and $\EE(e^k_\mathrm{erg})$, respectively, along the same line of arguments, it is straightforward to show
$\EE|E^k_\mathrm{erg}|=\mathcal{O}(1/k)$ and $\EE(e^k_\mathrm{erg})=\mathcal{O}(1/k)$.
\end{proof}

\subsection{Linear convergence analysis}
We first review some definitions and  inequities. Let $h$ be a closed convex proper function.
We let $\mu_h\ge 0$  be the strong-convexity constant of $h$,
where $\mu_h>0$ when $h$ is strongly convex and  $\mu_h=0$ otherwise.
When $h$ is differentiable and $\nabla h$ is Lipschitz continuous, we define
$(1/\beta_h)$ be the Lipschitz constant of $\nabla h$. When $h$ is either non-differentiable or differentiable but $\nabla h$ is not Lipschitz, we define $\beta_h=0$.
Under these definitions, we have
\begin{align}\label{eq:fubnd}
h(y) - &h(x) \nonumber\\
&\ge \langle \tnabla h(x),y-x\rangle + \underbrace{\frac{1}{2}\max\big\{\mu_h\|x-y\|^2,\beta_h\|\tnabla h(x)-\tnabla f(y) \|^2\big\}}_{S_h(x,y)},
\end{align}
for any points $x,y$ where the subgradients $\tnabla h(x),\tnabla h(y)$ exist. 
Note that $S_h(x,y)=S_h(y,x)$.

For the three convex functions $\brr,\brf,\brg$, we introduce their parameters $\mu_{\brr},\beta_{\brr},\mu_{\brf},\beta_{\brf},\mu_{\brg},\beta_{\brg}$, as well as the combination
\begin{align}
S(\vz,\vz^*)=S_{\brr}(\vx,\vx^* ) +S_{\brf}(\vx,\vx^* ) +S_{\brg}(\vx',\vx^* ).
\end{align}
As we assume each $f_i$ has $L$-Lipschitz gradient, we set $\beta_{\brf}=1/L$. Since $\brr$ includes the indicator function $I_{C}$, which is non-differentiable, we set $\beta_{\brr}=0$. The values of remaining parameters $\mu_{\brr},\mu_{\brf},\mu_{\brg},\beta_{\brg}\ge0$ are kept unspecified. Applying \eqref{eq:fubnd} to each pair of the three functions in $$E = \big(\brr(\vx)+\brf(\vx)+\brg(\vx')\big)-\big(\brr(\vx^*)+\brf(\vx^*)+\brg(\vx^*)\big)$$ yields
\begin{align}
E &\le (\tnabla\brr(\vx)+\nabla\brf(\vx))^T( \vx-\vx^* ) +
( \tnabla\brg(\vx'))^T(\vx'-\vx^*  ) -S(\vz,\vz^*),
\\
E &\ge (\tnabla\brr(\vx^*)+\nabla\brf(\vx^*))^T( \vx-\vx^* ) +
( \tnabla\brg(\vx^*))^T(\vx'-\vx^*  )+ S(\vz,\vz^*).
\end{align}
In this fashion, both the upper and lower bounds on  $E^k$, which we previously derive, are tightened by $S(\vz^k,\vz^*)$.
In particular, we can tightened  \eqref{eq:telscpEk} and \eqref{eq:keyEk2} as
\begin{align}
E^k & \le \tfrac{1}{2\alpha}\|\vz^k-\vz^*\|^2-\tfrac{1}{2\alpha}\|\vz^{k+1}-\vz^*\|^2-\tfrac{\alpha}{2}\|p(\vz^k)\|^2
+
\alpha ( \tnabla\brr(\vx^*)+\nabla\brf(\vx^{k+1/2}))^Tp(\vz^k)-S(\vz^k,\vz^*),\\
E^k & \ge (\tnabla \brr(\vx^*)+\nabla\brf(\vx^*))^T\alpha p(\vz^k) +S(\vz^k,\vz^*),
\end{align}
where the two terms involving $S(\vz^k,\vz^*)$ are newly added.
Combining the upper and lower bounds of $E^k$ yields
\begin{align}\label{eq:geomz}
\tfrac{1}{2\alpha}\|\vz^{k+1}-\vz^*\|^2 \le \tfrac{1}{2\alpha}\|\vz^k-\vz^*\|^2 - Q,
\end{align}
where
\begin{align}
Q &= -\alpha ( \tnabla\brr(\vx^*)+\nabla\brf(\vx^{k+1/2}))^Tp(\vz^k)+\alpha(\tnabla \brr(\vx^*)+\nabla\brf(\vx^*))^T p(\vz^k) +\tfrac{\alpha}{2}\|p(\vz^k)\|^2+2S(\vz^k,\vz^*)\\
&=-\alpha (\nabla\brf(\vx^{k+1/2})-\nabla\brf(\vx^*) )^Tp(\vz^k)+\tfrac{\alpha}{2}\|p(\vz^k)\|^2+2S(\vz^k,\vz^*)\\
&=\Big(-\alpha (\nabla\brf(\vx^{k+1/2})-\nabla\brf(\vx^*) )^Tp(\vz^k)+\tfrac{\alpha}{2}\|p(\vz^k)\|^2+2S_{\brf}(\vz^k,\vz^*)\Big) +2S_{\brr}(\vz^k,\vz^*) +2S_{\brg}(\vz^k,\vz^*)\\
&\ge c_1\Big(\|p(\vz^k)\|^2+\|\nabla\brf(\vx^{k+1/2})-\nabla\brf(\vx^*)\|^2\Big)+
\frac{\mu_{\brf}}{2}\|\vx^{k+1/2}-\vx^{\star}\|^2+2S_{\brr}(\vz^k,\vz^*) +2S_{\brg}(\vz^k,\vz^*),
\end{align}
where $c_1>0$ is a constant and the inequality follows from the Young's inequality:
\[
\alpha (\nabla\brf(\vx^{k+1/2})-\nabla\brf(\vx^*) )^Tp(\vz^k)\le
\tfrac{\alpha}{4}\|p(\vz^k)\|^2+\alpha  \|\nabla\brf(\vx^{k+1/2})-\nabla\brf(\vx^*)\|^2
\]
and $\alpha < \beta_{\brf}$.
Later on, in three different cases, we will show
\begin{align}\label{eq:Qk}
Q\ge C\|\vz^k-\vz^*\|^2.
\end{align}
Hence, by substituting \eqref{eq:Qk} into \eqref{eq:geomz},  we  obtain the Q-linear (or quotient-linear) convergence relation
\begin{align}
\|\vz^{k+1}-\vz^*\|\le \sqrt{1-2\alpha C} \|\vz^k-\vz^*\|,
\end{align}
from which it is easy to further derive the Q-linear convergence results for $|E^k|$ and $e^k$.

\vspace{0.1in}\noindent\textbf{Case 1.} Assume $\brg$ is both strongly convex and has Lipschitz gradient, i.e., $\mu_{\brg},\beta_{\brg}>0$,  (and $\brf$ still has Lipschitz gradient). In this case,
\begin{align}
Q &= c_1\Big(\|p(\vz^k)\|^2+\|\nabla\brf(\vx^{k+1/2})-\nabla\brf(\vx^*)\|^2\Big)+2S_{\brg}(\vz^k,\vz^*)\\
&\ge c_1\Big(\|p(\vz^k)\|^2+\|\nabla\brf(\vx^{k+1/2})-\nabla\brf(\vx^*)\|^2\Big)+ \mu_{\brg} \|\vx^{k+1}-\vx^{\star}\|^2. \label{eq:c1q}
\end{align}
(Here we use $\nabla \brg$ instead of $\tnabla \brg$ since $\brg$ is differentiable.)
By the identities
\begin{align}
\vz^k &= \vx^{k+1}-\alpha  \big(\nabla\brg(\vx^{k+1})+ \nabla \brf(\vx^{k+1/2})\big)+2\alpha p(\vz^k),\\
\vz^* &= \vx^{\star}  \quad-\alpha  \big(\nabla\brg(\vx^{\star})+ \nabla \brf(\vx^{\star})\big),
\end{align}
 the triangle inequality, and $\|\nabla\brg(\vx^{k+1})-\nabla\brg(\vx^{\star})\|\le (1/\beta_{\brg})\| \vx^{k+1} - \vx^{\star} \|$, we get
\begin{align}
\|\vz^k-\vz^*\|^2 & =\big\|(\vx^{k+1}-\vx^{\star})-\alpha\big(\nabla\brg(\vx^{k+1})-\nabla\brg(\vx^{\star})\big)-\alpha\big(\nabla\brf(\vx^{k+1/2})-\nabla\brf(\vx^{\star})\big)+2\alpha p(\vz^k)\big\|^2\\
&\le 3\big\|(\vx^{k+1}-\vx^{\star})-\alpha\big(\nabla\brg(\vx^{k+1})-\nabla\brg(\vx^{\star})\big)\big\|^2+3\alpha^2\|\nabla\brf(\vx^{k+1/2})-\nabla\brf(\vx^{\star})\|^2+12\alpha^2\| p(\vz^k)\|^2\\
&\le 3(1+\tfrac{\alpha}{\beta_{\brg}})^2\|\vx^{k+1}-\vx^{\star}\|^2 +3\alpha^2\|\nabla\brf(\vx^{k+1/2})-\nabla\brf(\vx^{\star})\|^2+12\alpha^2\| p(\vz^k)\|^2 \label{eq:c1z}.
\end{align}
Since \eqref{eq:c1z} is bounded by \eqref{eq:c1q} up to a constant factor, we have established \eqref{eq:Qk} for this case.

\vspace{0.1in}\noindent\textbf{Case 2.}  $\brf$ is strongly convex and $\brg$ has Lipschitz gradient, i.e., $\mu_{\brf}, \beta_{\brg}>0$ (and $\brf$ still has Lipschitz gradient).
In this case, 
\begin{align}
Q &= c_1\Big(\|p(\vz^k)\|^2+\|\nabla\brf(\vx^{k+1/2})-\nabla\brf(\vx^*)\|^2\Big)+\mu_{\brf}\|\vx^{k+1/2}-\vx^*\|^2.\label{eq:c2q}
\end{align}
By the identities
\begin{align}
\vz^k &= \vx^{k+1/2}-\alpha  \big(\nabla\brg(\vx^{k+1})+ \nabla \brf(\vx^{k+1/2})\big)+\alpha p(\vz^k),\\
\vz^* &= \vx^{\star}  \qquad-\alpha  \big(\nabla\brg(\vx^{\star})+ \nabla \brf(\vx^{\star})\big),
\end{align}
 the triangle inequality, and $\|\nabla\brg(\vx^{k+1})-\nabla\brg(\vx^{\star})\|\le (1/\beta_{\brg})\| \vx^{k+1} - \vx^{\star} \|$, we get
\begin{align}
\|\vz^k-\vz^*\|^2 & =\big\|(\vx^{k+1/2}-\vx^{\star})-\alpha\big(\nabla\brg(\vx^{k+1})-\nabla\brg(\vx^{\star})\big)-\alpha\big(\nabla\brf(\vx^{k+1/2})-\nabla\brf(\vx^{\star})\big)+\alpha p(\vz^k)\big\|^2\\
&\le 4\|\vx^{k+1/2}-\vx^{\star}\|^2 +4\alpha^2\|\nabla\brg(\vx^{k+1})-\nabla\brg(\vx^{\star})\|^2+4\alpha^2\|\nabla\brf(\vx^{k+1/2})-\nabla\brf(\vx^{\star})\|^2+4\alpha^2\| p(\vz^k)\|^2\\
&\le 4\|\vx^{k+1/2}-\vx^{\star}\|^2 +4\tfrac{\alpha^2}{\beta_{\brg}^2}\|\vx^{k+1}-\vx^{\star}\|^2+4\alpha^2\|\nabla\brf(\vx^{k+1/2})-\nabla\brf(\vx^{\star})\|^2+4\alpha^2\| p(\vz^k)\|^2\\
&\le 4(1+\tfrac{2\alpha^2}{\beta_{\brg}^2})\|\vx^{k+1/2}-\vx^{\star}\|^2+4\alpha^2\|\nabla\brf(\vx^{k+1/2})-\nabla\brf(\vx^{\star})\|^2+4\alpha^2(1+\tfrac{2\alpha^2}{\beta_{\brg}^2})\| p(\vz^k)\|^2  \label{eq:c2z},
\end{align}
where the last line follows from $$\|\vx^{k+1}-\vx^{\star}\|^2=\|\vx^{k+1/2}-\vx^{\star}-\alpha p(\vz^k)\|^2\le 2\|\vx^{k+1/2}-\vx^{\star} \|^2+ 2\alpha^2\| p(\vz^k)\|^2.$$
Since \eqref{eq:c2z} is bounded by \eqref{eq:c2q} up to a constant factor, we have established \eqref{eq:Qk} for this case.

\vspace{0.1in}\noindent\textbf{Case 3.} $\brr$ is strongly convex and $\brg$ has Lipschitz gradient, in short, $\mu_{\brr}\beta_{\brg}>0$, (and $\brf$ still has Lipschitz gradient).
In this case, 
\begin{align}
Q &= c_1\Big(\|p(\vz^k)\|^2+\|\nabla\brf(\vx^{k+1/2})-\nabla\brf(\vx^*)\|^2\Big)+\mu_{\brr}\|\vx^{k+1/2}-\vx^*\|^2.\label{eq:c3q}
\end{align}
We still use \eqref{eq:c2z}, which is bounded by \eqref{eq:c3q} up to a constant factor, we have established \eqref{eq:Qk} for this case.

\section{Conclusion and future work}
In this paper we presented
\eqref{eq:main-method} and a variant \eqref{eq:main-sppg}.
By discussing possible applications,
we demonstrated how \eqref{eq:main-method}
expands the class of optimization problems
that can be solved with a simple and scalable method.
We proved convergence and demonstrated the effectiveness,
especially in parallel computing,
of the methods through computational experiments.



An interesting future direction
is to consider cyclic and asynchronous variations of
\eqref{eq:main-method}.
Generally speaking, random coordinate updates
can be computationally inefficient, and
cyclic coordinate updates access the data more efficiently.
\eqref{eq:main-method} is a synchronous algorithm; at each iteration
the $z_1,\dots,z_n$ are updated synchronously and the synchronization
can cause inefficiency.
Asynchronous updates avoid this problem.







\bibliographystyle{plain}
\bibliography{ppg_siam,wyin}

\begin{thebibliography}{10}

\bibitem{arjevani2015}
T.~Arjevani and O.~Shamir.
\newblock Communication complexity of distributed convex learning and
  optimization.
\newblock In {\em NIPS}, pages 1756--1764. 2015.

\bibitem{baillon1977}
J.-B. Baillon and G.~Haddad.
\newblock Quelques propri\'et\'es des op\'erateurs angle-born\'es et
  $n$-cycliquement monotones.
\newblock {\em Israel Journal of Mathematics}, 26(2):137--150, 1977.

\bibitem{bauschke2010}
H.~H. Bauschke and P.~L. Combettes.
\newblock {\em The {B}aillon-{H}addad Theorem Revisited}, 17(3--4):781--787,
  2010.

\bibitem{bauschke2011}
H.~H. Bauschke and P.~L. Combettes.
\newblock {\em Convex Analysis and Monotone Operator Theory in Hilbert Spaces}.
\newblock 2011.

\bibitem{bertsekas2011}
D.~P. Bertsekas.
\newblock Incremental proximal methods for large scale convex optimization.
\newblock {\em Mathematical Programming}, 129(2):163--195, 2011.

\bibitem{bianchi2016}
P.~Bianchi.
\newblock Ergodic convergence of a stochastic proximal point algorithm.
\newblock {\em SIAM Journal on Optimization}, 26(4):2235--2260, 2016.

\bibitem{bien2013}
J.~Bien, J.~Taylor, and R.~Tibshirani.
\newblock A lasso for hierarchical interactions.
\newblock {\em The Annals of Statistics}, 41(3):1111--1141, 2013.

\bibitem{boyd2011}
S.~Boyd, N.~Parikh, E.~Chu, B.~Peleato, and J.~Eckstein.
\newblock Distributed optimization and statistical learning via the alternating
  direction method of multipliers.
\newblock {\em Foundations and Trends in Machine Learning}, 3(1):1--122, 2011.

\bibitem{chandrasekaran2011}
V.~Chandrasekaran, S.~Sanghavi, P.~A. Parrilo, and A.~S. Willsky.
\newblock Rank-sparsity incoherence for matrix decomposition.
\newblock {\em SIAM Journal on Optimization}, 21(2):572--596, 2011.

\bibitem{chen2012}
X.~Chen, Q.~Lin, S.~Kim, J.~G. Carbonell, and E.~P. Xing.
\newblock Smoothing proximal gradient method for general structured sparse
  regression.
\newblock {\em The Annals of Applied Statistics}, 6(2):719--752, 2012.

\bibitem{combettes2011}
P.~L. Combettes and J.-C. Pesquet.
\newblock {\em Proximal Splitting Methods in Signal Processing}, pages
  185--212.
\newblock 2011.

\bibitem{combettes2015_stoch}
P.~L. Combettes and J.-C. Pesquet.
\newblock Stochastic quasi-{{Fej{\'e}r}} block-coordinate fixed point
  iterations with random sweeping.
\newblock {\em SIAM Journal on Optimization}, 25(2):1221--1248, 2015.

\bibitem{combettes2005}
P.~L. Combettes and V.~R. Wajs.
\newblock Signal recovery by proximal forward-backward splitting.
\newblock {\em Multiscale Modeling and Simulation}, 4(4):1168--1200, 2005.

\bibitem{cortes1995}
C.~Cortes and V.~Vapnik.
\newblock Support-vector networks.
\newblock {\em Machine Learning}, 20(3):273--297, 1995.

\bibitem{davis2017}
Damek Davis and Wotao Yin.
\newblock A three-operator splitting scheme and its optimization applications.
\newblock {\em Set-Valued and Variational Analysis}, Jun 2017.

\bibitem{defazio2014b}
A.~Defazio, F.~Bach, and A.~Lacoste-Julien.
\newblock {SAGA}: A fast incremental gradient method with support for
  non-strongly convex composite objectives.
\newblock In {\em NIPS}, pages 1646--1654. 2014.

\bibitem{defazio2014}
A.~Defazio, J.~Domke, and T.~S. Caetano.
\newblock {Finito: A faster, permutable incremental gradient method for big
  data problems}.
\newblock In {\em {ICML}}, volume~32, pages 1125--1133, 2014.

\bibitem{deng2016}
W.~Deng, M.-J. Lai, Z.~Peng, and W.~Yin.
\newblock Parallel multi-block {ADMM} with $o(1/k)$ convergence.
\newblock {\em Journal of Scientific Computing}, 2016.

\bibitem{fan2008}
R.-E. Fan, K.-W. Chang, C.-J. Hsieh, X.-R. Wang, and C.-J. Lin.
\newblock {LIBLINEAR}: A library for large linear classification.
\newblock {\em Journal of Machine Learning Research}, 9:1871--1874, 2008.

\bibitem{gabay1976}
D.~Gabay and B.~Mercier.
\newblock A dual algorithm for the solution of nonlinear variational problems
  via finite element approximation.
\newblock {\em Computers and Mathematics with Applications}, 2(1):17--40, 1976.

\bibitem{glowinski1975}
R.~Glowinski and A.~Marroco.
\newblock Sur l'approximation, par \'el\'ements finis d'ordre un, et la
  r\'esolution, par p\'enalisation-dualit\'e d'une classe de problèmes de
  {D}irichlet non lin\'eaires.
\newblock {\em Revue Fran\c{c}aise d'Automatique, Informatique, Recherche
  Op\'erationnelle. Analyse Num\'erique}, 9(2):41--76, 1975.

\bibitem{hallac2015}
D.~Hallac, J.~Leskovec, and S.~Boyd.
\newblock Network lasso: Clustering and optimization in large graphs.
\newblock In {\em KDD}, pages 387--396, 2015.

\bibitem{hoefling2010}
H.~Hoefling.
\newblock A path algorithm for the fused lasso signal approximator.
\newblock {\em Journal of Computational and Graphical Statistics},
  19(4):984--1006, 2010.

\bibitem{jaggi2014}
M.~Jaggi, V.~Smith, M.~Takac, J.~Terhorst, S.~Krishnan, T.~Hofmann, and M.~I.
  Jordan.
\newblock Communication-efficient distributed dual coordinate ascent.
\newblock In {\em NIPS}. 2014.

\bibitem{jenatton2011}
R.~Jenatton, J.-Y. Audibert, and F.~Bach.
\newblock Structured variable selection with sparsity-inducing norms.
\newblock {\em Journal of Machine Learning Research}, 12:2777--2824, 2011.

\bibitem{johnson2013}
R.~Johnson and T.~Zhang.
\newblock Accelerating stochastic gradient descent using predictive variance
  reduction.
\newblock In {\em NIPS}, pages 315--323. 2013.

\bibitem{kim2010}
S.~Kim and E.~P. Xing.
\newblock Tree-guided group lasso for multi-task regression with structured
  sparsity.
\newblock In {\em ICML}, pages 543--550, 2010.

\bibitem{kulis2010}
B.~Kulis and P.~L. Bartlett.
\newblock Implicit online learning.
\newblock In {\em ICML}, pages 575--582, 2010.

\bibitem{lan2015}
G.~Lan and Y.~Zhou.
\newblock An optimal randomized incremental gradient method.
\newblock 2015.

\bibitem{langford2009}
J.~Langford, L.~Li, and T.~Zhang.
\newblock Sparse online learning via truncated gradient.
\newblock {\em Journal of Machine Learning Research}, 10:777--801, 2009.

\bibitem{leroux2012}
N.~{Le Roux}, M.~Schmidt, and F.~Bach.
\newblock Stochastic gradient method with an exponential convergence rate for
  finite training sets.
\newblock In {\em NIPS}, 2012.

\bibitem{jun2010}
J.~Liu and J.~Ye.
\newblock Moreau-{Y}osida regularization for grouped tree structure learning.
\newblock In {\em NIPS}, pages 1459--1467. 2010.

\bibitem{liu2010}
J.~Liu, L.~Yuan, and J.~Ye.
\newblock An efficient algorithm for a class of fused lasso problems.
\newblock In {\em KDD}, pages 323--332. 2010.

\bibitem{mairal2013}
J.~Mairal.
\newblock Optimization with first-order surrogate functions.
\newblock In {\em ICML}, pages 783--791, 2013.

\bibitem{mairal2015}
J.~Mairal.
\newblock Incremental majorization-minimization optimization with application
  to large-scale machine learning.
\newblock {\em SIAM Journal on Optimization}, 25(2):829--855, 2015.

\bibitem{mairal2010}
J.~Mairal, R.~Jenatton, F.~R. Bach, and G.~R. Obozinski.
\newblock Network flow algorithms for structured sparsity.
\newblock In {\em NIPS}, pages 1558--1566. 2010.

\bibitem{mccullagh1989}
P.~McCullagh and J.~A. Nelder.
\newblock {\em Generalized linear models}.
\newblock 2nd edition, 1989.

\bibitem{mcmahan2011}
H.~B. McMahan.
\newblock Follow-the-regularized-leader and mirror descent: Equivalence
  theorems and l1 regularization.
\newblock In {\em AISTATS}, 2011.

\bibitem{minty1962}
G.~J. Minty.
\newblock Monotone (nonlinear) operators in {H}ilbert space.
\newblock {\em Duke Mathematical Journal}, 29(3):341--346, 1962.

\bibitem{mota2013}
J.~F.~C. Mota, J.~M.~F. Xavier, P.~M.~Q. Aguiar, and M.~P\"uschel.
\newblock {D-ADMM}: A communication-efficient distributed algorithm for
  separable optimization.
\newblock {\em IEEE Transactions on Signal Processing}, 61(10):2718--2723,
  2013.

\bibitem{nitanda2014}
A.~Nitanda.
\newblock Stochastic proximal gradient descent with acceleration techniques.
\newblock In {\em NIPS}, pages 1574--1582. 2014.

\bibitem{nowak2011}
G.~Nowak, T.~Hastie, J.~R. Pollack, and R.~Tibshirani.
\newblock A fused lasso latent feature model for analyzing multi-sample {aCGH}
  data.
\newblock {\em Biostatistics}, 12(4):776--791, 2011.

\bibitem{palomar2007}
D.~P. Palomar and M.~Chiang.
\newblock Alternative distributed algorithms for network utility maximization:
  Framework and applications.
\newblock {\em IEEE Transactions on Automatic Control}, 52(12):2254--2269,
  2007.

\bibitem{parikh2014}
N.~Parikh and S.~Boyd.
\newblock Proximal algorithms.
\newblock {\em Foundations and Trends in Optimization}, 1(3):127--239, 2014.

\bibitem{passty1979}
G.~B. Passty.
\newblock Ergodic convergence to a zero of the sum of monotone operators in
  {H}ilbert space.
\newblock {\em Journal of Mathematical Analysis and Applications},
  72(2):383--390, 1979.

\bibitem{raguet2013}
H.~Raguet, J.~Fadili, and G.~Peyr\'e.
\newblock A generalized forward-backward splitting.
\newblock {\em SIAM Journal on Imaging Sciences}, 6(3):1199--1226, 2013.

\bibitem{rapaport2008}
F.~Rapaport, E.~Barillot, and J.-P. Vert.
\newblock Classification of {arrayCGH} data using fused {SVM}.
\newblock {\em Bioinformatics}, 24(13):i375--i382, 2008.

\bibitem{robbins1985}
H.~Robbins and D.~Siegmund.
\newblock A convergence theorem for non negative almost supermartingales and
  some applications.
\newblock In {\em Herbert {{Robbins Selected Papers}}}, pages 111--135. 1985.

\bibitem{ryu2014}
E.~K. Ryu and S.~Boyd.
\newblock Stochastic proximal iteration: A non-asymptotic improvement upon
  stochastic gradient descent.
\newblock 2014.

\bibitem{ryu2016}
E.~K. Ryu and S.~Boyd.
\newblock Primer on monotone operator methods.
\newblock {\em Applied and Computational Mathematics}, 15(1):3--43, 2016.

\bibitem{salim2016}
A.~Salim, P.~Bianchi, W.~Hachem, and J.~Jakubowicz.
\newblock A stochastic proximal point algorithm for total variation
  regularization over large scale graphs.
\newblock In {\em IEEE CDC}, pages 4490--4495, 2016.

\bibitem{schmidt2016}
M.~Schmidt, N.~{Le Roux}, and F.~Bach.
\newblock Minimizing finite sums with the stochastic average gradient.
\newblock {\em Mathematical Programming}, pages 1--30, 2016.

\bibitem{shalev2016}
S.~Shalev-Shwartz.
\newblock {SDCA} without duality, regularization and individual convexity.
\newblock In {\em ICML}, pages 747--754, 2016.

\bibitem{shalev2103}
S.~Shalev-Shwartz and T.~Zhang.
\newblock Stochastic dual coordinate ascent methods for regularized loss.
\newblock {\em Journal of Machine Learning Research}, 14:567--599, 2013.

\bibitem{shalev2016b}
S.~Shalev-Shwartz and T.~Zhang.
\newblock Accelerated proximal stochastic dual coordinate ascent for
  regularized loss minimization.
\newblock {\em Mathematical Programming}, 155(1):105--145, 2016.

\bibitem{shamir2014}
O.~Shamir, N.~Srebro, and T.~Zhang.
\newblock Communication-efficient distributed optimization using an approximate
  {N}ewton-type method.
\newblock In {\em ICML}, 2014.

\bibitem{tibshirani2005}
R.~Tibshirani, M.~Saunders, S.~Rosset, J.~Zhu, and K.~Knight.
\newblock Sparsity and smoothness via the fused lasso.
\newblock {\em Journal of the Royal Statistical Society: Series B},
  67(1):91--108, 2005.

\bibitem{tibshirani2008}
R.~Tibshirani and P.~Wang.
\newblock Spatial smoothing and hot spot detection for {CGH} data using the
  fused lasso.
\newblock {\em Biostatistics}, 9(1):18, 2008.

\bibitem{toulis2017}
P.~Toulis and E.~M. Airoldi.
\newblock Asymptotic and finite-sample properties of estimators based on
  stochastic gradients.
\newblock {\em Annals of Statistics (To appear)}, 2017.

\bibitem{toulis2014}
P.~Toulis, J.~Rennie, and E.~M. Airoldi.
\newblock Statistical analysis of stochastic gradient methods for generalized
  linear models.
\newblock In {\em ICML}, pages 667--675, 2014.

\bibitem{wang2015}
M.~Wang and D.~P. Bertsekas.
\newblock Incremental constraint projection methods for variational
  inequalities.
\newblock {\em Mathematical Programming}, 150(2):321--363, 2015.

\bibitem{wang2016}
M.~Wang and D.~P. Bertsekas.
\newblock Stochastic first-order methods with random constraint projection.
\newblock {\em SIAM Journal on Optimization}, 26(1):681--717, 2016.

\bibitem{xiao2014}
L.~Xiao and T.~Zhang.
\newblock A proximal stochastic gradient method with progressive variance
  reduction.
\newblock {\em SIAM Journal on Optimization}, 24(4):2057--2075, 2014.

\bibitem{yang2014}
E.~Yang, A.~Lozano, and P.~Ravikumar.
\newblock Elementary estimators for high-dimensional linear regression.
\newblock In {\em ICML}, pages 388--396, 2014.

\bibitem{yang2013}
T.~Yang.
\newblock Trading computation for communication: Distributed stochastic dual
  coordinate ascent.
\newblock In {\em NIPS}. 2013.

\bibitem{ye2011}
G.-B. Ye and X.~Xie.
\newblock Split {B}regman method for large scale fused lasso.
\newblock {\em Computational Statistics \& Data Analysis}, 55(4):1552--1569,
  2011.

\bibitem{yuan2011}
L.~Yuan, J.~Liu, and J.~Ye.
\newblock Efficient methods for overlapping group lasso.
\newblock In {\em NIPS}, pages 352--360. 2011.

\bibitem{yuan2006}
M.~Yuan and Y.~Lin.
\newblock Model selection and estimation in regression with grouped variables.
\newblock {\em Journal of the Royal Statistical Society: Series B},
  68(1):49--67, 2006.

\bibitem{zhang2013}
L.~Zhang, M.~Mahdavi, and R.~Jin.
\newblock Linear convergence with condition number independent access of full
  gradients.
\newblock In {\em NIPS}, pages 980--988. 2013.

\bibitem{zhang20152}
Y.~Zhang and L.~Lin.
\newblock Disco: Distributed optimization for self-concordant empirical loss.
\newblock In {\em ICML}, 2015.

\bibitem{zhang2012}
Y.~Zhang, M.~J. Wainwright, and J.~C. Duchi.
\newblock Communication-efficient algorithms for statistical optimization.
\newblock In {\em NIPS}, pages 1502--1510. 2012.

\bibitem{zhao2009}
P.~Zhao, G.~Rocha, and B.~Yu.
\newblock The composite absolute penalties family for grouped and hierarchical
  variable selection.
\newblock {\em The Annals of Statistics}, (6A):3468--3497, 2009.

\bibitem{zhou2012}
J.~Zhou, J.~Liu, V.~A. Narayan, and J.~Ye.
\newblock Modeling disease progression via fused sparse group lasso.
\newblock In {\em KDD}, pages 1095--1103. 2012.

\end{thebibliography}
\end{document}